\documentclass[11pt,letterpaper]{amsart}
\usepackage{amsfonts}
\usepackage{amsmath,amssymb,color}
\usepackage{graphicx}
\usepackage{enumitem}
\usepackage[doublespacing]{setspace}

\usepackage{breakcites}

\usepackage[margin=1in]{geometry}

\newcommand{\CC}{\mbox{$\mathbf C$}}

\newcommand{\EE}{\mbox{$\mathbf E$}}

\newcommand{\J}{\mbox{$\mathbf J$}}

\newcommand{\QQ}{\mbox{$\mathbf Q$}}

\newtheorem{theorem}{Theorem}
\newtheorem*{assumption}{Assumption}

\newtheorem{algorithm}{Algorithm}

\newtheorem{corollary}{Corollary}

\newtheorem{lemma}{Lemma}

\newtheorem{proposition}{Proposition}

\usepackage{hyperref}
\usepackage{amsmath,amssymb,amsfonts,bm}
\usepackage{amsfonts}
\usepackage{mathrsfs}
\usepackage{xcolor}
\usepackage{graphicx}
\usepackage[margin=1in]{geometry}
\usepackage{caption,subcaption}

\usepackage{empheq}

\usepackage[norelsize,linesnumbered,vlined,ruled,algo2e]{algorithm2e}

\newcommand{\dom}{{U}}

\newcommand{\Real}{\mathbb{R}}

\newcommand{\E}{\mathbb{E}}
\DeclareMathOperator\Var{Var}
\DeclareMathOperator\Cov{Cov}
\newcommand{\g}{{\mathcal{G}}}
\newcommand{\dg}{{\mathcal{G}'}}
 \newcommand{\Fd}{Fr\'{e}chet derivative}
 
 \newcommand{\PP}{\mathbb{P}}
\newcommand{\C}{\mathbf{C}}
\newcommand{\ud}{\,\mathrm{d}}
\newcommand{\K}{{\mathcal{K}}}
\newcommand{\h}{{\mathcal{H}}}
 \renewcommand{\QQ}{\mathbb{Q}}

\begin{document}

\title{Moderate Deviation for Random Elliptic PDEs with Small Noise}
\author{Xiaoou Li}
\address{School of Statistics, University of Minnesota, 
	Minneapolis, MN 55455}
\email{lixx1766@umn.edu}
\author{Jingchen Liu}
\address{Department of Statistics,  Columbia University, New York, NY 10027
}
\email{jcliu@stat.columbia.edu}
\author{Jianfeng Lu}

\address{Department of Mathematics,
    Department of Physics, and Department of Chemistry, Duke
    University, Durham, NC 27708 USA}
\email{jianfeng@math.duke.edu}

\author{Xiang Zhou}
\address{Department of Mathematics,
City University of Hong Kong,
Tat Chee Ave, Kowloon, 
Hong Kong SAR}
\email{xizhou@cityu.edu.hk}

\date{\today}

\maketitle

\baselineskip=18pt

\begin{abstract}
	Partial differential equations with random inputs have become popular models to characterize physical systems with uncertainty coming from, e.g., imprecise measurement and intrinsic randomness.
	In this paper, we perform asymptotic rare event analysis for such elliptic PDEs with random inputs.
	In particular, we consider the asymptotic regime that the noise level converges to zero suggesting that the system uncertainty is low, but does exists. We develop sharp approximations of the probability of a large class of rare events.
	
%
\end{abstract}

\section{Introduction}

The study of rare events due to system uncertainty, for example the
failure of materials due to intrinsic randomness, is crucial and yet
challenging. While those events do not often occur, they lead to
catastrophic consequences. Therefore it is important to estimate the
probabilities of such events and to characterize those events which 
help finding interventions to prevent them from happening. In this paper, 
we consider the  following classical continuum mechanical model in the form of a linear
elliptic partial differential equation (PDE)
defined on  a domain $\dom\subset \Real^d$,
\begin{equation}\label{eqn:EllipticPDE}
  - \nabla \cdot (a(x)\nabla u(x)) = f(x),
\end{equation}
subject to   certain boundary conditions that will be 
specified in the sequel. 
The solution to the above equation $u$ is the displacement field of the elastic material,
$\nabla u$ is the strain, $a$ is the elasticity tensor,  $a(x) \nabla u(x)$ is the stress tensor,  and  $f$ is the external body force.
The elasticity tensor $a(x)$ (which is uniformly positive definite)
is determined by the property of the specific material.
Instead of assuming that $a$ is deterministic, we are interested in the situations when
the tensor $a$ contains randomness. The randomness is introduced to incorporate the uncertainties  of simple elastic materials at the macroscopic level or heterogeneity in
the microstructures of complex materials.
Under this setting, the solution $u(x)$ (as a function of $a(x)$) is also a stochastic process whose law is determined by that of $a(x)$.

Besides material mechanics,  the elliptic PDE \eqref{eqn:EllipticPDE} arises 
also in many other fields of applications,
such as  hydrogeology and porous medium. 
The tensor $a(x)$ carries different names
such as conductivity and permeability. It is recognized that
the modeling of the random field $a(x)$ is of primal importance for
the analysis.
In this paper, we consider  that the random function $a(x)$ follows a log-normal distribution,
that is, 
\begin{equation}\label{a-sigma}
  a(x) = a_0(x)e^{ -  \sigma \xi(x)} \qquad x \in \dom,
\end{equation}
where $\xi(x)$ is a Gaussian random field defined on $\dom$ and
$a_0(x)$ is a deterministic function. In elasticity, in
  general $a(\cdot)$ is a function of $4$-tensor. For simplicity of
  notation, we consider a scalar field here (i.e., an isotropic
  material). The technique and result for a general $a(\cdot)$ is
  similar.  The scalar $\sigma>0$ is a parameter indexing the noise
level.  Many studies by practioners, e.g.,~\cite{Freeze-lognormal,
  Bear1987Modeling-Ground,Charbeneau2000}, have shown that the best
fit of the empirical data is the log-normal distribution. Hence, the
log-normal assumption is well justified in applications and is used in
mathematical analysis and numerical computation of the random PDE
\eqref{eqn:EllipticPDE}. In our paper, we follow this convention of
log-normal assumption for the rare-event analysis.

In this work, we consider the small noise asymptotic regime, that $\sigma$ tends to zero. Yet, even small
noise can lead to drastic difference of the PDE solution from that of
the deterministic case when the noise level is zero. Our results
characterize such rare events, more precisely, the 
deviation of the solution of the random elliptic PDE in the presence of
small noise.
In particular, we focus on the deviation from the deterministic solution   as the uncertainty level goes to $0$.
 Let $\mathcal{H}$ be 
 a mapping from $C(\bar{\dom})$ to $\Real$. Of primary interest
 \[
 \omega(\sigma) = \PP\{ \mathcal{H}(u)> \mathcal{H}(u_0) + b_\sigma\} \quad \text{as }  \sigma\to 0.
 \]
where $u$ is the solution to  equation \eqref{eqn:EllipticPDE}
and $u_0$ is the solution when the noise level is zero, i.e., $a(x)=a_0(x)$.
The level $b_\sigma$ will be sent to zero as the noise level $\sigma$ goes to zero,
which will be specified  in the sequel. The main contribution 
of this paper is to derive sharp asymptotic approximations of $\omega(\sigma)$
as $\sigma\to 0$.

Given that   $\mathcal{H}(u)$ is  a (complicated) functional of the
input  Gaussian process $\xi(x)$,
the analysis of the tail probability $\omega(\sigma)$   links naturally to the rare-event analysis of Gaussian random field.
The study of the extremes of  Gaussian random fields focuses mostly on
the tail probabilities of the supremum of the field. The results contain
general bounds on $P(\max \xi(x)>b)$ as well as sharp asymptotic approximations as $%
b\rightarrow \infty$. A partial literature contains \cite%
{LS70,MS70,ST74,Borell1975,Bor03,LT91,TA96,Berman85}.
Several methods have been introduced to obtain bounds and asymptotic
approximations. A general upper bound for the tail of $\max \xi(x)$ is developed in \cite%
{Borell1975,CIS}, which is known as the Borel--TIS inequality. For asymptotic results,
there are several methods, such as the double sum method (\cite{Pit95}) ,
 the Euler--Poincar\'{e} characteristics
of the excursion set approximation (\cite%
{Adl81,TTA05,AdlTay07,TayAdl03}),   the tube method (\cite{Sun93}),
and  the Rice method (\cite{AW08,AW09}).
Recently, the exact tail approximation of integrals of exponential functions of Gaussian random fields is developed by \cite{Liu10,LiuXu11}. Efficient computations via importance sampling has been developed by \cite{ABL08,ABL09}.
For the analysis of the tail probabilities of lognormal random fields with small noise,
refer to the recent work in \cite{LiLiuXu2016}.
There are also existing work in the context of PDE with random coefficients.
\cite{LiuZhou13,LiuZhou14}  derive asymptotic analysis
of one-dimensional elliptic PDE. \cite{SISC2015} presents the corresponding
rare-event simulation algorithms. These works focused on
the asymptotic regime that the noise level $\sigma$ is fixed.
Furthermore, \cite{KdVUQ2014} presents asymptotic analysis 
for stochastic KdV equation.

The rest of the paper is organized as follows. 
Section \ref{sec:main} presents 
the problem setup and the main
asymptotic  results.
The technical  proofs  are given in Section~\ref{sec:proof}.

\section{Main results}\label{sec:main}

\subsection{The problem setup}\label{sec:setup}
We consider the following elliptic  PDE.
Let $\dom \subset \Real^d$ be an open domain with a smooth boundary. 
The differential equation concerning $u: \dom \to \Real$ with Dirichlet boundary condition is given by
\begin{equation}\label{eq:PDE}
  \begin{cases}
    - \nabla \cdot (a(x) \nabla u (x)) = f(x) & \text{for  } x\in\dom;  \\
    u (x)= 0 & \text{for } x\in\partial \dom.
  \end{cases}
\end{equation}
In the context of elastic  mechanics, $u$ characterizes the material
deformation due to external force $f$ and $a: \dom \to \Real$ gives the
stiffness of the material.  Throughout this paper, we assume $u(\cdot)$ to be a scalar function for simplicity. 
 We
assume that the material is clamped to a frame on the boundary
$\partial \dom$ and hence the Dirichlet boundary condition $u(\partial
\dom) = 0$ in \eqref{eq:PDE} is assumed. The external force $f$ is sufficiently smooth and 
bounded, that is, there exists a constant $c\in \Real$ such that
\begin{equation}
  \label{eq:fbound}
  |f(x)| \leq c , \qquad \forall x\in\dom.
\end{equation}
We study the behavior of the material under the influence of internal randomness, which may be 
the result of manufacturing  processing or the uncertainty of the material properties at the microscopic level. We adopt a probabilistic viewpoint of the complexity and heterogeneity inherent in the material and   view the   coefficient $a(x)$ as a random field.
The process $a(x)$ is physically restricted to be positive and  is modeled as a lognormal random field given as 
  in \eqref{a-sigma}. 
Furthermore, the Gaussian random function $\xi$ has mean zero and 
its  covariance
function  is denoted by
\begin{equation}\label{cov}
C(x,y) = \EE \{\xi(x) \xi(y)\},
\end{equation}
which is certainly independent of $\sigma$. In addition,  $C$ admits the normalization condition
 $C(x,x)\equiv 1$.

The solution $u(x)$ depends implicitly on $a(x)$  through equation \eqref{eq:PDE} 
and further $\xi(x)$ via  a logarithmic change of variable.
It is useful to  define a mapping from the coefficient $\xi$ to the solution $u$ 
 \[  \J[\xi] \triangleq u_\xi  \]
where $u_\xi$ is the solution to equation \eqref{eq:PDE} with $a(x)=a_0(x)e^{-\xi(x)}$.
This mapping  depends only on the   deterministic function $a_0$,  the external force $f$, the domain 
$\dom$, and the boundary condition. In this paper, we are interested in the asymptotic regime that the amplitude of the uncertainty level $\sigma$
tends to zero. Then
 the failure problem concerns the random solution $u_{\sigma \xi}= \J (\sigma \xi)$
 by noting the definition of $\J$ above. As $\sigma\to 0$, the process $a(x)$ tends to its limiting field $a_0(x)$. Let $u_0(x)$ be the corresponding limiting solution satisfying equation
\begin{equation}\label{u0}
  \begin{cases}
    - \nabla \cdot (a_0(x) \nabla u_0 (x)) = f(x) & \text{for  } x\in\dom;  \\
    u_0 (x)= 0 & \text{for } x\in\partial \dom.
  \end{cases}
\end{equation}
Then,  under mild conditions, we have $u(x) \to u_0(x)$ as $\sigma\to 0$.

We provide asymptotic analysis of the event that $u$ deviates from its limiting solution $u_0$.
Let $\h$ be a functional  from $C(\bar \dom)$ to $\Real$ characterizing the deviation. 
For instance $\h(u)= \int_{\dom} (u(x)-u_0(x) )~dx$. 
Let $\g$ be the composition of $\J$ and $\h$, that is, $$\g(\xi) = \h(\J[\xi]).$$
To simplify notation, we always choose $\h$ such that $\g(\mathbf 0) = \h (u_0) =0$.
We are interested the tail probability of $\g(\sigma\xi)$ as $\sigma \to 0$.
In particular, we derive asymptotic approximations for 
\begin{equation}\label{prob}
\omega(\sigma) = \PP\{\g(\sigma \xi)> b\} \quad \text{as }  \sigma\to 0,
\end{equation}
where the deviation level is chosen to be $b = \kappa \sigma ^\alpha$
for some fixed $\alpha \in (0,1)$ and $\kappa>0$. In particular, the
deviation level $b$ also goes to $0$ as the uncertainty vanishes.

\subsection{Asymptotic results}
We first introduce some notation that will be used in the sequel.
Throughout this analysis, we consider $\g$ to be a differentiable function  and let $\dg$ be its \Fd, 
that is, 
\begin{equation*}
\g(\xi  + \varepsilon \eta) = \g(\xi) + \varepsilon \int_\dom
\dg[\xi] (x) \eta(x) dx  + o(\varepsilon),
\quad \mbox {as $\varepsilon \to 0$}, ~~\forall\, \xi, \eta \in C(\bar \dom).
\end{equation*}
For $0<\beta< 1$, we say that a function $w$ is H\"older continuous with order $\beta$ if the H\"older coefficient 
\begin{equation}\label{eq:holder-seminorm}
[w]_{\beta}=\sup_{x,y\in \bar{\dom},x\neq y}\frac{|w(x)-w(y)|}{|x-y|^{\beta}}  < \infty.
\end{equation}
We use  $C^{k}(\bar \dom)$ to denote the space containing all $k$-time continuously  differentiable functions. For nonnegative integer $k$ and $0\leq \beta<1$, we use $C^{k,\beta}(\bar \dom)$ to denote the  set of functions in $C^{k}(\bar \dom)$ whose $k$-th order partial derivatives are H\"older continuous with coefficient $\beta$.
For simplicity, we write
$
C^{0,\beta}(\bar \dom)=C^{\beta}(\bar \dom)$.
We proceed to the definition of norms over $C^{k,\beta}(\bar \dom)$.
We first define the seminorms
$$
[w]_{k,0}=\max_{|\gamma|=k}\sup_{\bar \dom}|D^{\gamma}w|\quad
\mbox{ and } \quad [w]_{k,\beta}= \max_{|\gamma|=k}~[D^{\gamma}w]_{\beta},
$$
where $\gamma$ is a multi-index $\gamma=(\gamma_1,...,\gamma_d)$, $|\gamma|=\sum_{i=1}^d \gamma_i$, and $D^{\gamma}w=\frac{\partial^{|\gamma|}w}{\partial^{\gamma_1}x_1...\partial^{\gamma_d}x_d}$.
We further define the norms
$$
\|w\|_{C^{k}(\bar \dom)}=\sum_{j=0}^k [w]_{j,0} \quad\mbox{ and }\quad \|w\|_{C^{k,\beta}(\bar \dom)} = \|w\|_{C^{k}(\bar \dom)} +[w]_{k,\beta}.
$$
Equipped with $\|\cdot\|_{C^{k,\beta}(\bar \dom)}$, the space $C^{k,\beta}(\bar \dom)$ is a Banach space for all non-negative integer $k$ and $0\leq \beta<1$.
To simplify notation, we write
$$
|w|_{k}=\|w\|_{C^{k}(\bar \dom)} , \quad |w|_{k,\beta}= \|w\|_{C^{k,\beta}(\bar \dom)},\quad    |w|_{\beta}=|w|_{0,\beta}.
$$

 We now present sharp asymptotic approximations of the tail probabilities $w(\sigma)$ under  the following assumptions on the functional $\g$ and the covariance function $C(x,y)$.
\begin{assumption}\
\begin{itemize}
\item[A1.] There exist constants $k,\beta,\delta_G,\kappa_G$ such that
  $k$ is a non-negative integer, $0\leq\beta< 1$, $\delta_G>0$ and for
  all $|w|_{k,\beta}\leq \delta_G$, $\g'[w]\in C^{k,\beta}(\bar
  \dom)$. In addition $\g'$ is a (local) Lipschitz operator in the sense that for
  all $|w_1|_{k,\beta},|w_2|_{k,\beta}\leq \delta_G$, we have
	$$
	\bigl\lvert\g'[w_1]-\g'[w_2] \bigr\rvert_{k,\beta}\leq \kappa_G| w_1-w_2|_{k,\beta}.
	$$
   \item[A2.] There exists $x\in \bar \dom$ such that
    $\g'[\mathbf{0}](x)\neq 0$.
  \item[A3.] The Gaussian random field $\{\xi(x):x\in \dom\}$ has a  H\"older
    continuous sample path and belongs to the space $C^{k,\beta}(\bar
    \dom)$ almost surely, that is, $\PP(|\xi|_{k,\beta}<\infty)=1$.
    The covariance function $C(\cdot,\cdot)$ is positive
    definite and satisfies  $\sup_{y\in\bar
      \dom}|C(\cdot,y)|_{k,2\beta}<\infty$. Moreover, we assume that
    $\sup_{y\in\bar \dom} |C_{D^{\gamma}\xi}(\cdot,y)|_{2\beta}<\infty $
    for all $\gamma$ such that $|\gamma|\leq k$, where we define
	\begin{equation}\label{eq:def-cd}
	C_{D^{\gamma}\xi}(x,y) \triangleq  \E\{ D^{\gamma}\xi(x)D^{\gamma}\xi(y)\}.
	\end{equation}
\end{itemize}
\end{assumption}
Define a mapping $\C: C(\bar \dom) \to C(\bar \dom)$
$$\C w \triangleq \int C(\cdot, y) w(y) \ud y.$$
We consider the optimization problem
\begin{equation}{ \label{opt:b}}
	\min_{\xi  \in \mathcal{B}, 	\g(\sigma   \C{\xi} ) = b}   \K(\xi)
\end{equation}
where the functional $\K:C^{0}(\bar \dom)\to \mathbb{R}$ is
$$\K(w) \triangleq \int_\dom w(x)C(x,y)w(y)dx dy,$$
and the set $\mathcal{B}$ is defined as
\begin{equation}\label{eq:B}
\mathcal{B} \triangleq  \left\{
w\in C^{k,\beta}(\bar \dom):|
w
|_{k,\beta}\leq \sigma^{\alpha-1-\varepsilon}
\right\}
\end{equation}
for some $\varepsilon >0$ and $\alpha$ is given as below \eqref{prob}.
Because $\mathcal{B}$ is a compact subset of $C^{k,\beta}(\bar \dom)$ and the functionals $\K$ and $\g$ are continuous over $\mathcal{B}$, the above optimization problem has at least one solution. Later in the current section, we will show that this solution is also unique. With the above optimization, we have the following sharp asymptotic approximation for the tail probability of $\omega(\sigma)$.
\begin{theorem}\label{thm:small}
	Under Assumptions A1-A3, for $0<\alpha<1$ and $b=\kappa \sigma^{\alpha}$, we have
	$$
	\PP\{\g(\sigma \xi) > b\} = (c_1+o(1))\sigma^{1-\alpha}\exp\Big(-\frac{1}{2}K^*_{\sigma}\Big)\mbox{ as }\sigma\to 0,
	$$
	where $c_1=\kappa^{-1}\{(2\pi)^{-1}\K(\g'[\mathbf{0}]) \}^{\frac{1}{2}}$ and
	$$K^*_{\sigma} = \min_{w  \in \mathcal{B}, 	\g(\sigma   \C{w} ) = b}   \K(w).$$
\end{theorem}

	The constants $k$ and $\beta$ in Assumptions A1-A3 are problem-dependent. For example,  \cite{li2015tail} consider the functional
	$$\g(\xi)=\int_\dom e^{\sigma \xi(t)+\mu(t)}dt - \int_{\dom}e^{\mu(t)}dt,$$
	where $\mu(\cdot) \in C^0(\bar \dom)$ is a deterministic function. This particular $\g$ satisfies Assumptions A1 and A2 with $k=0$ and $\beta=0$.  In the context of elliptic PDE, the following theorem presents sufficient conditions for Assumptions A1-A3 with $k=1$ and $0<\beta<1$.

\begin{theorem}\label{thm:PDE-A1}
Let the functional  $\g(\xi)=\h(u_{\xi})$, where $u_{\xi}$ is the solution to \eqref{eq:PDE}. Suppose that the following assumptions hold.
	 \begin{itemize}
	 	\item [H1.] There exist constants $\beta,\delta_H,\kappa_H$ such that
	 	$\delta_H>0$, $0<\beta<1$ and  $\h'(u)\in C^{\beta}(\bar \dom)$ for all $|u-u_0|_{2,\beta}\leq \delta_G$. In addition, $\h'$ is Lipschitz in the sense that
	 	$$
	 	|\h'[u_1]-\h'[u_2] |_{\beta}\leq \kappa_H| u_1-u_2|_{2,\beta}
	 	$$
		for all $|u_1-u_0|_{2,\beta},|u_2-u_0|_{2,\beta}\leq \delta_H$.
	 	Here, $u_0\in C^{2,\beta}(\bar \dom)$ is the solution to \eqref{u0} when $\xi$ is set to be $\mathbf{0}$.
	 	\item[H2.] There exists $x\in \bar \dom$ such that $\nabla g_0(x)\cdot\nabla u_0(x)\neq 0$, where $g_0\in C^{2,\beta}(\bar \dom)$ is the solution to the PDE \begin{equation}\label{eq:g_0-PDE}
	 	\begin{cases}
	 	-  \nabla \cdot( a_0(x)\nabla g_0(x))
	 	= \h'[u_0](x) & \text{for  } x\in\dom;  \\
	 	g_0 (x)= 0 & \text{for } x\in\partial \dom,
	 	\end{cases}
	 	\end{equation}
	 	
	 	\item[H3.] $\dom$ is a bounded domain with a $C^{2,\beta}$ boundary $\partial \dom$,  $a_0\in C^{1,\beta}(\bar \dom)$, $\min_{x\in \bar \dom}a_0(x)>0$ and $f\in C^{\beta}(\bar \dom)$.
	 \end{itemize}
	 \begin{itemize}
     \item [H4.] The Gaussian random field $\{\xi(x),x\in \dom\}$ is
       H\"older continuous and belongs to the space $C^{k,\beta}(\bar
       \dom)$ almost surely.  Its covariance function $C(\cdot,\cdot)$ is
       positive definite and satisfies $\sup_{y\in\bar
         \dom}|C(\cdot,y)|_{1,2\beta}<\infty$. Moreover, we assume that
       $\sup_{y\in\bar \dom} |C_{D^{\gamma}\xi}(\cdot,y)|_{2\beta}<\infty
       $ for all $\gamma$ such that $|\gamma|\leq 1$, where
       $C_{D^{\gamma}\xi}$ is defined in \eqref{eq:def-cd}.
	 \end{itemize}
Then Assumptions A1-A3 are satisfied with $k=1$ and the H\"older coefficient being $\beta$.	
\end{theorem}
Under Assumption H3, the PDE \eqref{eq:PDE} has a unique solution $u_0\in C^{2,\beta}(\bar \dom)$ when $\xi$ is set to be $\mathbf{0}$.
Furthermore, under Assumptions H1 and H3, \eqref{eq:g_0-PDE} also has a unique solution in $C^{2,\beta}(\bar \dom)$. Therefore, $g_0$ and $u_0$ in the above theorem are well defined.
See Lemma~\ref{lemma:regularity} on page \pageref{lemma:regularity} for the existence and 
the uniqueness of the H\"older continuous solution to elliptic PDEs.
Combining Theorems ~\ref{thm:small} and \ref{thm:PDE-A1}, we arrive at the next corollary.
\begin{corollary}
Under the assumptions of Theorem~\ref{thm:PDE-A1}, for $0<\alpha<1$ and $b=\kappa \sigma^{\alpha}$, we have
	$$
	\PP\{\g(\sigma \xi) > b\} = (c_2+o(1))\sigma^{1-\alpha}\exp\Big(-\frac{1}{2}K^*_{\sigma}\Big)\mbox{ as }\sigma\to 0,
	$$
	where $c_2=\kappa^{-1}\{(2\pi)^{-1}\K(a\nabla g_0\cdot\nabla u_0) \}^{\frac{1}{2}}$ and $K^*_{\sigma}$ is the minimum obtained in \eqref{opt:b}.
\end{corollary}

\subsection{Numerical approximation}

Now we proceed to characterizing the solution to the optimization~\eqref{opt:b}.
\begin{theorem}\label{prop:opt}
Under Assumptions A1-A3, \begin{itemize}
	\item[(i)] the optimization problem \eqref{opt:b} has a unique solution for $\sigma$ sufficiently small, denoted by $\xi^*$;
	\item[(ii)] we have the following approximation as $\sigma\to 0$
	$$
	\xi^*=(1+o_{k,\beta}(1))\kappa \sigma^{\alpha-1}\frac{\g'[\mathbf{0}]}{\K(\g'[\mathbf{0}])},
	$$
	where we write $h_{\sigma}(\cdot)=o_{k,\beta}(1)$ if $|h_{\sigma}|_{k,\beta}=o(1)$ as $\sigma\to 0$.
\end{itemize}
\end{theorem}

The solution of the optimization in \eqref{opt:b}
is generally not in a closed form. 
Theorem \ref{prop:opt}
presents its first order approximation.
It is not accurate enough for a sharp asymptotic approximation.
We present further a  numerical approximation  for  $\xi^*$
in the following section.

 In this section, we present a numerical method for computing the solution $\xi^*$ to \eqref{opt:b}.
To solve theoptimization, we introduce the Lagrangian multiplier $\lambda\in \Real$ and  define the Lagrangian function $L$
$$L(\xi) = \iint \xi(x) C(x,y) \xi(y) dx dy - 2\frac{\lambda}{\sigma} (\g(\sigma \C{\xi}) - b).$$
The first order condition $\frac{\partial L}{\partial \xi}  \equiv 0$
implies  the KKT condition for $\lambda$  and $\xi$
\[\C \xi =\lambda  \C  \dg[\sigma \C{\xi}].\]
Since the covariance function $C(x,y)$ is positive definite and thus the linear map  $\C$ is a bijection. The above condition becomes
\begin{equation}
  \xi =  \lambda  \dg[\sigma \C{\xi}].
\end{equation}
The solution ($\xi^*$, $\lambda^*$) to the constrained optimization problem
is determined   by
\begin{subequations}{\label{KKT:w}}
  \begin{empheq}[left={ \empheqlbrace\,}]{align}
	&\xi^* = \lambda^*    \dg[\sigma \C{\xi^*}] ,
	\label{kkt1:w}
	\\
	&\g(\sigma \C \xi^*)=b.
	\label{kkt2:w}
  \end{empheq}
\end{subequations}
Our strategy  is to first find $\lambda$
given $\xi$ to satisfy the constraint \eqref{kkt2:w}; and then we look
for $\xi$ and the corresponding $\lambda = \Lambda(\xi)$ determined by
the previous step to satisfy the fix point equation \eqref{kkt1:w}.
Motivated by this, we define a functional
$$
\Lambda:\mathcal{B}\to [-\sigma^{\alpha-1-\varepsilon},\sigma^{\alpha-1-\varepsilon}]
$$
such that for each $w\in\mathcal{B}$, $\lambda=\Lambda(w)$ solves the following equation
\begin{equation}\label{eq:lam}
\g(\sigma\C \lambda \g'[\sigma\C w])=b. ~~
\end{equation}
To see that $\Lambda(\cdot)$ is well defined, for each $w\in \mathcal{B}$ we define the function $T_w:[-\sigma^{\alpha-1-\varepsilon},\sigma^{\alpha-1-\varepsilon}]\to \mathbb{R}$,
$$
T_w(\lambda)= \lambda - \K(\g'[\mathbf{0}])^{-1}\sigma^{-1} \Big(\g(\sigma\C \lambda \g'[\sigma\C w])-b \Big).
$$
Clearly,  solutions to \eqref{eq:lam} are fixed points of the function $T_w(\cdot)$. The well-posedness of the function $\Lambda(\cdot)$ is then established by the next proposition.
\begin{proposition}\label{prop:contract-t}
	For $\sigma$ sufficiently small, $w\in \mathcal{B}$, and $|\lambda_1|,|\lambda_2|\leq \sigma^{\alpha-1-\varepsilon}$, we have that  $|T_w(\lambda_1)|,|T_w(\lambda_2)|\leq \sigma^{\alpha-1-\varepsilon} $ and there exists a constant $\kappa_T$ independent of $\sigma$ and $w$, such that
$$
|T_w(\lambda_1)-T_w(\lambda_2)|\leq  \kappa_T \sigma^{\alpha-\varepsilon}|\lambda_1-\lambda_2|.
$$
\end{proposition}

The above proposition and the contraction mapping theorem guarantee that for each $w\in\mathcal{B}$, $T_{w}(\cdot)$ has a unique fixed point in $[-\sigma^{\alpha-1-\varepsilon},\sigma^{\alpha-1-\varepsilon}]$. Therefore, there is a unique solution $\Lambda[w]\in[-\sigma^{\alpha-1-\varepsilon},\sigma^{\alpha-1-\varepsilon}]$ satisfying \eqref{eq:lam}. Furthermore, it ensures the convergence of the iterative algorithm based on the contraction mapping $T_w(\lambda)$.
We further define an operator $\Xi$.
\begin{equation}\label{eq:Xi}
\Xi[w]=\Lambda[w]\g'[\sigma\C w].
\end{equation}

\begin{proposition}\label{prop:contract-w}
	For $\sigma$ sufficiently small,
$\Xi$ is a contraction mapping over $\mathcal{B}$. More specifically, there exists a constant $\kappa_{\Xi}$ such that for all $w_1,w_2\in \mathcal{B}$, we have
$$
|\Xi[w_1]-\Xi[w_2]|_{k,\beta}\leq \kappa_{\Xi}\sigma^{\alpha}|w_1-w_2|_{k,\beta}.
$$
\end{proposition}
The above proposition and the contraction mapping theorem guarantee that \eqref{KKT:w} has a unique solution $(\lambda^*,\xi^*)$ in $[-\sigma^{\alpha-1-\varepsilon},\sigma^{\alpha-1-\varepsilon}]\times \mathcal{B}$. Furthermore, this solution can be computed numerically via the following iterative algorithm.
\begin{enumerate}
	\item[1.] Initialize $\hat{\xi}^*_0=\kappa \sigma^{\alpha-1}\frac{\g'[\mathbf{0}]}{\K(\g'[\mathbf{0}])}.$
	\item[2.] At $l$-th iteration, update $\hat{\xi}^*_{l}$ by
	$$
	\hat{\xi}^*_{l}=\Xi[\hat{\xi}^*_{l-1}].
	$$

\end{enumerate}
According to the contraction mapping theorem, the rate of convergence is
$$
|\hat{\xi}^*_l-\xi^*|_{k,\beta}\leq (\kappa_{\Xi}\sigma^{\alpha})^{l}|\hat{\xi}^*_0-\xi^*|_{k,\beta}=O(\sigma^{\alpha l+\alpha-1}).
$$
Therefore, if we run $l>\frac{2(1-\alpha)}{\alpha}$ iterations, then $|\hat{\xi}^*_l-\xi^*|_{k,\beta}=o(\sigma^{1-\alpha})$, and we could use $\K(\hat{\xi}^*_l)$ to approximate $K^*_{\sigma}$ in Theorem~\ref{thm:small}.

\section{Technical proofs}\label{sec:proof}

Throughout the proof we will use $\kappa_0$ as generic notation for large and not-so-important constants whose value may vary from place to place. Similarly, we use $\varepsilon_0$ as generic notation for small positive constants. Furthermore, for two sequences $a_{\sigma}$ and $b_{\sigma}$, we write $a_{\sigma}=o(b_{\sigma})$ if $b_{\sigma}/a_{\sigma}\to 0$ as $\sigma$ tend to zero and $a_{\sigma}=O(b_{\sigma})$ if $b_{\sigma}/a_{\sigma}$ is bounded when $\sigma$ varies. Moreover, for two sequences of functions $a_{\sigma}(\cdot)$ and $b_{\sigma}(\cdot)$, we write
$a_{\sigma}=o_{k,\beta}(b_{\sigma})$ if $|a_{\sigma}|_{k,\beta}=o(|b_{\sigma}|_{k,\beta})$ and  $a_{\sigma}=O_{k,\beta}(b_{\sigma})$ if $|a_{\sigma}|_{k,\beta}=O(|b_{\sigma}|_{k,\beta})$.

The proofs in this sections are organized as follows. The proof of Theorem~\ref{thm:small} is presented in Section~\ref{sec:proof-thm-small}. Section~\ref{sec:proof-pde} shows the proof of Theorem~\ref{thm:PDE-A1}. Section~\ref{sec:proof-prop} presents proofs of Proposition~\ref{prop:contract-t}, \ref{prop:contract-w}, and \ref{prop:opt}. The proofs of supporting lemmas are postponed to Appendix~\ref{sec:proof-lemmas}.

\subsection{Proof of Theorem \ref{thm:small}}\label{sec:proof-thm-small}

We start with a useful lemma that restrict our analysis on the event $\mathcal{L}=\{\xi-\C \xi^*\in \mathcal{B}\}$, whose proof will be presented in Section~\ref{sec:proof-lemmas}.
\begin{lemma}\label{lemma:localization}
	There exists positive constant $\varepsilon_0$ such that
$$
\PP(\xi-\C \xi^*\in\mathcal{B}^c)\leq e^{-\varepsilon_0 \sigma^{2\alpha-2-2\varepsilon}}.
$$
\end{lemma}
\begin{proof}[Proof for Theorem \ref{thm:small}]
		Let $\xi^*$ be the solution to \eqref{opt:b}. We define an exponential change of
		measure
		\begin{equation}
		\frac{d\QQ}{d\PP}=\exp\Big(\int_{\dom} \xi^*(x) \xi(x)dx- \frac{1}{2}\int_\dom\int_\dom \xi^*(x)
		C(x, y) \xi^*(y)ds dt\Big).
		\end{equation}
		Under measure $\QQ$,  $\xi(x)$ is a Gaussian random field with mean function $\mathbf{C}\xi^*(x)$ and
		covariance function $C(x,y)$.
Let $$\mathcal L = \{\xi-\C \xi^*\in \mathcal{B}\}.$$
 According to Lemma~\ref{lemma:localization}, we only need to consider the event restricted to $\mathcal{L}$.
By means of the change of measure $\QQ$, we have
		\begin{eqnarray}\label{eqproof1}
		&&\PP\left(\g(\sigma\xi)>b , \mathcal{L}\right)\notag\\
		&=& \EE^\QQ\left[\frac{d\PP}{d\QQ};~\g(\sigma\xi)>b,
		~\mathcal{L}\right]\notag\\
		&=&\exp\left(\frac{1}{2}\int_{\dom\times \dom}\xi^*(x) C(x, y) \xi^*(y)ds dt\right)
		\EE^{\QQ}\left[e^{-\int_\dom \xi^*(x)\xi(x)dx};\g(\sigma\xi)>b,
		\mathcal{L}\right],
		\end{eqnarray}
		where $\EE^\QQ$ denotes the expectation with respect to the measure $\QQ$.
		It is easy to check that the random field $\mathbf{C}\xi^*(x)+\xi(x)$ under $\PP$ has the same distribution as $\xi(x)$ under $\QQ$.
		Thus, we replace the probability measure $\QQ$ and $\xi$ with $\PP$ and $\mathbf{C}\xi^*+\xi$ in \eqref{eqproof1} and obtain
		\begin{eqnarray}\label{eqproof1'}
		&&\PP\left( \g(\sigma\xi)>b , \mathcal{L}\right)\notag\\
		&=&\exp\left(\frac{1}{2}\int_{\dom\times \dom}\xi^*(x) C(x, y) \xi^*(y)dx dy\right)
		\EE\left[e^{-\int_\dom \xi^*(x)(\mathbf{C}\xi^*(x)+\xi(x))dx};\g(\sigma(\xi+\mathbf{C}\xi^*) )>b,
		\xi\in \mathcal{B}\right]\notag\\
		&=&
		\exp\left(-\frac{1}{2}\int_{\dom\times \dom} \xi^*(x) C(x, y) \xi^*(y)ds dt\right)
		\EE\left[e^{-\int_\dom \xi^*(x)\xi(x)dx}; \g(\sigma(\xi+\mathbf{C}\xi^*) )-\g(\sigma\mathbf{C}\xi^*)>0,
	\xi\in\mathcal{B}\right]\notag\\
	&=& e^{-\frac{1}{2}K_{\sigma}^*}\times\EE\left[e^{-\int_\dom \xi^*(x)\xi(x)dx}; \g(\sigma(\xi+\mathbf{C}\xi^*) )-\g(\sigma\mathbf{C}\xi^*)>0,
	\xi\in\mathcal{B}\right].\notag
		\end{eqnarray}
		We define two events $$F=\{\g(\sigma(\xi+\mathbf{C}\xi^*) )-\g(\sigma\mathbf{C}\xi^*)>0\},  \mbox{ and }
		F_1= \Big\{\int_\dom\g'[\sigma \mathbf{C}\xi^*](x)\sigma\xi(x)dx>0\Big\}.
		$$	
		Let the event $\mathcal{L}_1= \{
		\xi\in\mathcal{B}\}$.
		We will present an approximation for
	$$
		I_1=	\EE\left[e^{-\int_\dom \xi^*(x)\xi(x)dx}; F_1 \right]
	$$
	and show that
	$$
		I_2=		\EE\left[e^{-\int_\dom \xi^*(x)\xi(x)dx}; (F_1\triangle F) \cap
				\mathcal{L}_1\right]
	$$
is ignorable, where $``\triangle"$ denotes the symmetric difference between two sets.
First, we compute
		\begin{equation}\label{exp1}
		I_1=	\EE\left[e^{-\int_\dom \xi^*(x) \xi(x)dx}; \int_\dom\g'[\sigma \mathbf{C}\xi^*](x)\xi(x)dx>0\right].
		\end{equation}
		According to Proposition \ref{prop:contract-w},  $\xi^*$ is the fixed point of the contraction map $\Xi$ and thus
		$$\xi^* = \Xi[\xi^*]=\Lambda[\xi^*]\g'[\sigma\C\xi^*].$$
		Therefore, $\xi^*$ and $\g'[\sigma \mathbf{C}\xi^*]$ are different only by a factor of $\Lambda[\xi^*]$.
		Thus, $\int_\dom \xi^*(x) \xi(x)dx$  and $\int_\dom\g'[\sigma \mathbf{C}\xi^*](x)\xi(x)dx>0$ are different by a factor $\Lambda[\xi^*]$.
		The following lemma establishes an approximation for $\Lambda[\xi^*]$.
		\begin{lemma}\label{lemma:approx-lam}
			For all $w\in\mathcal{B}$, $\Lambda[w]=\kappa \K(\g'[\mathbf{0}])^{-1}\sigma^{\alpha-1}(1+o(1))$. This approximation is uniform in $w$.
		\end{lemma}
		Thanks to Lemma \ref{lemma:approx-lam}, we have
		\begin{equation*}
		\Lambda[\xi^*]=(1+o(1))\frac{\kappa\sigma^{\alpha-1}}{\K(\g'[\mathbf{0}])}.
		\end{equation*}
		Let $Z_1=\int_\dom \xi^*(x)\xi(x)dx$, then $Z_1$ is a normally distributed random variable with a zero mean. The expectation \eqref{exp1}
		can be computed as follows
		\begin{eqnarray}
			&&\EE\left[e^{-Z_1};  Z_1>0, \right]\notag \\
			&=& \int_{0}^{\infty}\frac{1}{\sqrt{2\pi \Var(Z_1)}} e^{-\frac{z_1^2}{2\Var(Z_1)}-z_1}dz_1\notag\\
			&=& \frac{1}{\sqrt{2\pi \Var(Z_1)}} E[ e^{-\frac{V^2}{2\Var(Z_1)}}],\label{eq:expectation-exp}
		\end{eqnarray}
		where $V$ is a random variable following the exponential distribution with rate $1$.
Notice that
\begin{equation}\label{eq:var-z1}
\Var(Z_1)=\int_{\dom\times \dom} \xi^*(x)C(x,y)\xi^*(y)dxdy = (1+o(1))\kappa^2\sigma^{2\alpha-2}\K^{-1}[\g'[\mathbf{0}]].
\end{equation}
The second equality is obtained with the aid of Proposition~\ref{prop:opt}(ii).
The above display, \eqref{eq:expectation-exp} and dominated convergence theorem give
$$
I_1=\kappa^{-1}\{(2\pi)^{-1}\K(\g'[\mathbf{0}])\}^{1/2}\sigma^{1-\alpha}(1+o(1)).
$$
Now, we proceed to the term $I_2$.
\begin{lemma}\label{lemma:second-expansion}
	Under Assumption A1, we have that for $|w_1|_{k,\beta},|w_2|_{k,\beta}\leq \delta_G$,
	\begin{equation*}
	|w_1-w_2|_{k,\beta}^{-2}\Big|\g(w_1)-\g(w_2)-\int_\dom \g'[w_2](x)(w_1(x)-w_2(x))dx \Big|\leq meas(\dom)\kappa_G,
	\end{equation*}
	where $meas(\dom)$ is the Lebesgue measure of $\dom$ and  $k,\beta,\delta_G,\kappa_G$ are constants appeared in Assumption A1.
\end{lemma}
According to Lemma~\ref{lemma:second-expansion}, we have that for $\sigma$ sufficiently small and $\xi\in\mathcal{B}$,
\begin{equation}\label{eq:expand-cxi}
\Big|\g(\sigma(\xi+\C \xi^*))-\g(\sigma\C\xi^*)- \sigma\int_\dom \g'[\sigma\C\xi^*](x)\xi(x)dx\Big|\leq meas(\dom)\kappa_{G}\sigma^2|\xi|^2_{k,\beta}.
\end{equation}
Note that on the event $F_1\triangle F$, $\g(\sigma(\xi+\C \xi^*))-\g(\sigma\C\xi^*)$ and $\sigma\int_\dom \g'[\sigma\C\xi^*](x)\xi(x)dx$ have opposite signs and thus
\begin{equation}\label{eq:triangle-event}
\Big|\g(\sigma(\xi+\C \xi^*))-\g(\sigma\C\xi^*)- \sigma\int_\dom \g'[\sigma\C\xi^*](x)\xi(x)dx\Big|\geq |\sigma\int_\dom \g'[\sigma\C\xi^*](x)\xi(x)dx|.
\end{equation}
We combine \eqref{eq:expand-cxi} and \eqref{eq:triangle-event} and arrive at
$$
(F\triangle F_1)\cap \mathcal{L}_1 \subset\Big\{
meas(\dom)\kappa_G \|\xi\|^2_{k,\beta}\geq \sigma^{-1}|\int_\dom \g'[\sigma\C\xi^*](x)\xi(x)dx|
\Big\}\cap\mathcal{L}_1.
$$
We write $Z_2=\|\xi\|^2_{k,\beta}$, then the above display implies that
$$
(F\triangle F_1)\cap\mathcal{L}_1\subset\{
meas(\dom)\kappa_{G}Z_2\geq \sigma^{-1} \Lambda[\xi^*]| Z_1|
\}\cap{\mathcal{L}_1}.
$$
This gives an upper bound of the expectation
		$$\EE\left[e^{-\int_\dom \xi^*(x) \xi(x)dx}; (F\triangle F_1)\cap \mathcal{L}_1\right]\leq \EE\left[e^{-Z_1}; \kappa_{G}Z_2\geq  \sigma^{-1}\Lambda[\xi^*]|Z_1|, \mathcal{L}_1\right].$$
		On the event $\{0<|Z_1|\leq \sigma^{\varepsilon}\}$, this expectation is negligible compared to $I_1$, that is,
		\begin{equation}\label{major}
		E[e^{Z_1}; 0<|Z_1|<\sigma^{\varepsilon}]=O(\PP(0<|Z_1|<\sigma^{\varepsilon}))=O(\sigma^{1-\alpha+\varepsilon}).
		\end{equation}
		The second equality in the above display is due to \eqref{eq:var-z1}.
		Furthermore, on the set $\mathcal {L}_1$, we have $|Z_1|\leq |\xi^*|_0|\xi|_0 \kappa_0 \leq \kappa_0 \sigma^{2\alpha-2-\varepsilon}$, where $\kappa_0$ is a  sufficiently large constant. Therefore, we only need to focus on the expectation
		\begin{multline}\label{exp2}
		\EE\left[e^{Z_1}; \sigma^{\varepsilon}<|Z_1|<\kappa_0\sigma^{2\alpha-2-\varepsilon}, Z_2>\Lambda(\xi^*)|Z_1/\sigma|\right]= \int_{\sigma^{\varepsilon}}^{\kappa_0\sigma^{2\alpha-2-\varepsilon}} e^{z} \PP(Z_2>\Lambda(\xi^*)z/\sigma|Z_1=z) p_{Z_1}(z)dz\\
		 + \int_{\sigma^{\varepsilon}}^{\kappa_0\sigma^{2\alpha-2-\varepsilon}} e^{z} \PP(Z_2>\Lambda(\xi^*)z/\sigma|Z_1=-z) p_{Z_1}(z)dz,
		\end{multline}
		where $p_{Z_1}(z)$ is the density function of $Z_1$.
		\begin{lemma}\label{lemtail}
			For $z\in[\sigma^{\varepsilon},\kappa_0\sigma^{2\alpha-2-\varepsilon}]$, there exists a constant $\varepsilon_0>0$ such that
			\begin{equation}\label{conditionalprob}
			\PP(Z_2>\Lambda(\xi^*) z/\sigma|Z_1=z)+\PP(Z_2>\Lambda(\xi^*) z/\sigma|Z_1=-z)\leq e^{-\varepsilon_0\sigma^{\alpha-2} z}.
			\end{equation}
		\end{lemma}
		With the above lemma, the expectation \eqref{exp2} is bounded by
		\begin{eqnarray*}\label{minor}
		\eqref{exp2}&\leq &
		\int_{\sigma^{\varepsilon}}^{\kappa_0\sigma^{2\alpha-2-\varepsilon}} e^{-( \varepsilon_0\sigma^{\alpha-2}-1)z} p_{Z_1}(z)dz\\
		&=& \frac{1}{\sqrt{2\pi \Var(Z_1)}}\int_{\sigma^{\varepsilon}}^{\kappa_0\sigma^{2\alpha-2-\varepsilon}} e^{-( \varepsilon_0\sigma^{\alpha-2}-1)z-\frac{z^2}{2\Var(Z_1)}}dz\\
		&\leq&\frac{1}{\sqrt{2\pi \Var(Z_1)}} \int_{\sigma^{\varepsilon}}^{\kappa_0\sigma^{2\alpha-2-\varepsilon}} e^{-\frac{\varepsilon_0}{2}\sigma^{\alpha-2}z}dz,
		\end{eqnarray*}
		for $\sigma$ sufficiently small so that $\varepsilon_0\sigma^{\alpha-2}-1>\frac{\varepsilon_0}{2}\sigma^{\alpha-2}$. The above inequality is further bounded by
		$$
		\eqref{exp2}\leq \frac{1}{\sqrt{2\pi \Var(Z_1)}} \int_{\sigma^{\varepsilon}}^{\kappa_0\sigma^{2\alpha-2-\varepsilon}} e^{-\frac{\varepsilon_0}{2}\sigma^{\alpha-2}z}dz\leq \frac{1}{\sqrt{2\pi \Var(Z_1)}} \kappa_0\sigma^{2\alpha-2-\varepsilon} e^{-\frac{\varepsilon_0}{2}\sigma^{\alpha-2+\varepsilon}} = O(e^{-\frac{\varepsilon_0}{2}\sigma^{\alpha-2+\varepsilon}}).
		$$
		Therefore,
		$$
		\eqref{exp2}=o(\sigma^{1-\alpha}).
		$$
		We combine our analysis for $I_1$ and $I_2$ and conclude our proof for Theorem \ref{thm:small}.
	\end{proof}

\subsection{Proof of Theorem~\ref{thm:PDE-A1}}\label{sec:proof-pde}
\begin{proof}[Proof of Theorem~\ref{thm:PDE-A1}]
	We first present two useful lemmas. The following lemma guarantees the existence and uniqueness of the H\"older continuous solution to the elliptic PDE.
	\begin{lemma}\label{lemma:regularity}
		Suppose that $\dom$ is a bounded domain with a $C^{2,\beta}$ boundary $\partial \dom$ for $0<\beta<1$.
		Assume that there exist positive constants $\delta$ and $M$ such that $\min_{x\in \bar \dom} a(x)>\delta>0$, and $|a|_{1,\beta}\leq M$, and $f\in C^{\beta}(\bar \dom)$.  Then the elliptic PDE
		\begin{equation}\label{eq:gPDE}
		\begin{cases}
		- \nabla \cdot (a(x) \nabla u (x)) = f(x) & \text{for  } x\in\dom;  \\
		u (x)= 0 & \text{for } x\in\partial \dom,
		\end{cases}
		\end{equation}
		has a unique solution in $C^{2,\beta}(\bar \dom)$. Denote this solution by $u_{a,f}$, then
		\begin{equation}\label{eq:regularity-PDE}
		|u_{a,f}|_{2,\beta}\leq \kappa(\delta,M,d,\dom) |f|_{\beta},
		\end{equation}
		where $ \kappa(\delta,M,d,\dom)$ is a positive constant, depending only on $\delta,M,d$ and the domain $\dom$.
	\end{lemma}
	We will also need the following lemma on the stability of the solution.
	\begin{lemma}\label{lemma:stability}
				Suppose that $\dom$ is a bounded domain with a $C^{2,\beta}$ boundary $\partial \dom$ for $0<\beta<1$.
		Let $a_1$, $a_2$, $f_1$ and $f_2$ be functions over the domain $\dom$ such that
		$$
		\min_{x\in\bar \dom}a_1(x)\geq \delta,\qquad \min_{x\in\bar \dom}a_2(x)\geq \delta, \qquad |a_1|_{1,\beta}, |a_2|_{1,\beta}\leq M, \mbox{ and } f_1,f_2\in C^{\beta}(\bar \dom).
		$$
		Then,
		$$
		|u_{a_1,f_1}-u_{a_2,f_2}|_{2,\beta}\leq \tilde{\kappa}(\delta,M,d,\dom)\{|f_1-f_2|_{\beta}+|a_1-a_2|_{1,\beta}|f_1|_{\beta}\},
		$$
		where the constant $\tilde{\kappa}(\delta,M,d,\dom)$ depends only on $\delta,M,d$ and the domain $\dom$.
	\end{lemma}
	The Fr\'echet derivative $\g'[w]$ has the following expression.
	\begin{equation*}
	\g'[w](x)=a_{w}(x)\nabla g_{w}(x)\cdot\nabla u_w(x),
	\end{equation*}
	where $a_w(x)=a_0e^{-w(x)}$, $u_{w}\in C^{2,\beta}(\bar \dom)$ is the unique solution to
	\begin{equation*}
	\begin{cases}
	- \nabla \cdot (a_w(x) \nabla u_w (x)) = f(x) & \text{for  } x\in\dom;  \\
	u_w (x)= 0 & \text{for } x\in\partial \dom,
	\end{cases}
	\end{equation*}
	and
	$g_{w}(x)\in C^{2,\beta}(\bar \dom)$ is the unique solution to
	\begin{equation*}
	\begin{cases}
	- \nabla \cdot (a_w(x) \nabla g_w (x)) = \h'[u_w](x)& \text{for  } x\in\dom;  \\
	g_w (x)= 0 & \text{for } x\in\partial \dom.
	\end{cases}
	\end{equation*}
	For $w_1,w_2\in C^{1,\beta}(\bar \dom)$, we are going to establish an upper bound for $|\g'[w_1]-\g'[w_2]|_{1,\beta}$.
	Note that
	\begin{eqnarray*}
	&&\g'[w_1](x)-\g'[w_2](x)\notag\\
	&=& (a_{w_1}(x)-a_{w_2}(x))\nabla g_{w_1}(x)\cdot \nabla u_{w_1}(x)\notag\\
	&&+ a_{w_2}\nabla g_{w_2}(x)\nabla (u_{w_1}-u_{w_2}(x))+ a_{w_2}(x)\nabla (g_{w_1}(x)-g_{w_2}(x))\cdot \nabla u_{w_1}(x).
	\end{eqnarray*}
	Thus,
		\begin{eqnarray}
		&&|\g'[w_1]-\g'[w_2]\notag|_{1,\beta}\\
		&\leq& |(a_{w_1}-a_{w_2})\nabla g_{w_1}\cdot \nabla u_{w_1}|_{1,\beta}\notag\\
		&&+ |a_{w_2}\nabla g_{w_2}\nabla (u_{w_1}-u_{w_2})|_{1,\beta}+ |a_{w_2}\nabla (g_{w_1}-g_{w_2})\cdot \nabla u_{w_1}|_{1,\beta}.\label{eq:split-g'}
		\end{eqnarray}
	We will establish upper bounds for the three terms on the right-hand side in the above expression separately.
	First, note that $a_{w_k}=a_0 e^{-w_k}$, $k=1,2$. Thus, there exists a constant $\varepsilon_{0}>0$ such that for all $|w_1|_{1,\beta},|w_2|_{1,\beta}\leq \varepsilon_0$,
	\begin{equation}\label{eq:bound-diff-a}
	|a_{w_1}-a_{w_2}|_{1,\beta}\leq \kappa_0 |w_1-w_2|_{1,\beta}.
	\end{equation}
 Therefore,
	\begin{equation}\label{eq:split-g'-1}
	|(a_{w_1}-a_{w_2})\nabla g_{w_1}\cdot \nabla u_{w_1}|_{1,\beta}\leq |a_{w_1}-a_{w_2}|_{1,\beta}|\nabla g_{w_1}|_{1,\beta}|\nabla u_{w_1}|_{1,\beta}\leq \kappa_0 |w_1-w_2|_{1,\beta}|g_{w_1}|_{2,\beta}|u_{w_1}|_{2,\beta}.
	\end{equation}
	Now we present upper bounds for $|g_{w_1}|_{2,\beta}$ and $|u_{w_1}|_{2,\beta}$. Let $\varepsilon_0$ be  sufficiently small such that for all $|w|_{1,\beta}\leq \varepsilon_0$, $\min_{x\in \bar \dom} a_w(x)\geq \frac{1}{2}\min_{x\in\bar \dom}a_0(x)$ and $|a_{w}|_{1,\beta}\leq 2|a_0|_{1,\beta}$. According to Lemma~\ref{lemma:regularity}, we have that for all $|w|_{1,\beta}\leq \delta_0$
	\begin{equation}\label{eq:uw-bound}
	|u_{w}|_{2,\beta}\leq \kappa(\delta,M,d,\dom) |f|_{\beta},
	\end{equation}
	where $\delta=\frac{\min_{x\in\bar \dom}a_0(x)}{2}$ and $M=2|a_0|_{1,\beta}$.
	Furthermore, according to Assumption H1,
	we have that for $|u_w-u_0|_{2,\beta}\leq \delta_H$
	\begin{equation}\label{eq:h'-bound}
	|\h'[u_w]|_{\beta}\leq |\h'[u_0]|_{\beta}+ \kappa_{H}|u_w-u_0|_{2,\beta}\leq |\h'[u_0]|_{\beta}+ \kappa_{H}\delta_H.
	\end{equation}
 Set $f=\h'[u_w]$ in Lemma~\ref{lemma:regularity} we have
	\begin{equation}\label{eq:bound-gw}
	|g_{w}|_{2,\beta}\leq \kappa(\delta,M,d,\dom)	|\h'[u_w]|_{\beta}\leq \kappa(\delta,M,d,\dom)(|\h'[u_0]|_{\beta}+ \kappa_{H}\delta_H).
	\end{equation}
	Combine this with \eqref{eq:split-g'-1} and \eqref{eq:uw-bound}, we have that for $|w_1|_{1,\beta},|w_2|_{1,\beta}\leq \varepsilon_0$
	\begin{equation}\label{eq:bound-split-g'-1}
	|(a_{w_1}-a_{w_2})\nabla g_{w_1}\cdot \nabla u_{w_1}|_{1,\beta}\leq \kappa_0 |w_1-w_2|_{1,\beta},
	\end{equation}
	with a possibly different $\kappa_0$.
	We proceed to the second term on the right-hand side of \eqref{eq:split-g'}.
	\begin{equation}\label{eq:split-g'-2}
	|a_{w_2}\nabla g_{w_2}\cdot\nabla (u_{w_1}-u_{w_2})|_{1,\beta}\leq
	|a_{w_2}|_{1,\beta}|\nabla g_{w_2}|_{1,\beta} |\nabla(u_{w_1}-u_{w_2})|_{1,\beta}
	\leq |a_{w_2}|_{1,\beta}|g_{w_2}|_{2,\beta}|u_{w_1}-u_{w_2}|_{2,\beta}.
	\end{equation}
	For $|w_2|_{1,\beta}\leq \varepsilon_0$, we have $|a_{w_2}|_{1,\beta}\leq 2 |a_0|_{1,\beta}$. Moreover, $|g_{w_2}|_{2,\beta}$ is bounded above by a constant according to \eqref{eq:bound-gw}. Therefore,
	\begin{equation}\label{eq:bound1}
	|a_{w_2}\nabla g_{w_2}\cdot\nabla (u_{w_1}-u_{w_2})|_{1,\beta}\leq \kappa_0 |u_{w_1}-u_{w_2}|_{2,\beta},
	\end{equation}
	for a possibly different $\kappa_0$.
	Taking $a_1=a_{w_1}$, $a_2=a_{w_2}$, and $f_1=f_2=f$ in Lemma~\ref{lemma:stability}, we have
	\begin{equation}\label{eq:uw-dif-bound}
	|u_{w_1}-u_{w_2}|_{2,\beta}\leq\tilde{\kappa}(\delta,M,d,\dom) |a_1-a_2|_{1,\beta}|f_{body}|_{\beta}\leq \kappa_0 |w_1-w_2|_{1,\beta}.
	\end{equation}
	\eqref{eq:bound1} and \eqref{eq:uw-dif-bound} give
	\begin{equation}\label{eq:bound-split-g'-2}
	|a_{w_2}\nabla g_{w_2}\cdot\nabla (u_{w_1}-u_{w_2})|_{1,\beta}\leq\kappa_0^2 |w_1-w_2|_{1,\beta}.
	\end{equation}
	We proceed to the third term on the right-hand side of \eqref{eq:split-g'}.
	\begin{equation}\label{eq:split-g'-3}
	|a_{w_2}\nabla (g_{w_1}-g_{w_2})\cdot \nabla u_{w_1}|_{1,\beta}\leq |a_{w_2}|_{1,\beta}|\nabla (g_{w_1}-g_{w_2})|_{1,\beta} |\nabla u_{w_1}|_{1,\beta} \leq |a_{w_2}|_{1,\beta}|g_{w_1}-g_{w_2}|_{2,\beta}|u_{w_1}|_{2,\beta}.
	\end{equation}
	According to the definition of $a_{w_2}$ and \eqref{eq:uw-dif-bound}, we have that for $|w_1|_{1,\beta},|w_2|_{1,\beta}\leq \varepsilon_0$,
	\begin{equation}\label{eq:bound2}
	|a_{w_2}\nabla (g_{w_1}-g_{w_2})\cdot \nabla u_{w_1}|_{1,\beta} \leq\kappa_0 |g_{w_1}-g_{w_2}|_{2,\beta}.
	\end{equation}
	Motivated by the definition of $g_{w_1}$ and $g_{w_2}$, we take $f_1=\h'[w_1]$, $f_2=\h'[w_2]$, $a_1=a_{w_1}$ and $a_2=a_{w_2}$ in Lemma~\ref{lemma:stability}, then
	\begin{equation}\label{eq:diff-g}
	|g_{w_1}-g_{w_2}|_{2,\beta}\leq \tilde{\kappa}(\delta,M,d,\dom)\{|\h'[w_1]-\h'[w_2]|_{\beta}+|a_{w_1}-a_{w_2}|_{1,\beta}|\h'[w_1]|_{\beta}\}.
	\end{equation}
	According to Assumption H1, for $|w_1|_{1,\beta},|w_2|_{1,\beta}\leq \delta_H$, we have
	\begin{equation}\label{eq:H1-use}
	|\h'[w_1]-\h'[w_2]|_{\beta}\leq \kappa_H |w_1-w_2|_{1,\beta}
	\end{equation}
	\eqref{eq:H1-use}, \eqref{eq:bound-diff-a}, \eqref{eq:h'-bound} and \eqref{eq:diff-g} give
	\begin{equation*}
	|g_{w_1}-g_{w_2}|_{2,\beta}\leq\kappa_0 |w_1-w_2|_{1,\beta}.
	\end{equation*}
	The above inequality and \eqref{eq:bound2} give
	\begin{equation}\label{eq:bound-split-g'-3}
	|a_{w_2}\nabla (g_{w_1}-g_{w_2})\cdot \nabla u_{w_1}|_{1,\beta} \leq\kappa_0^2 |w_1-w_2|_{1,\beta}
	\end{equation}
	We combine \eqref{eq:split-g'}, \eqref{eq:bound-split-g'-1}, \eqref{eq:bound-split-g'-2}, and \eqref{eq:bound-split-g'-3}, and arrive at
	\begin{equation}
	|\g'[w_1]-\g'[w_2]|_{1,\beta}\leq \kappa_0 |w_1-w_2|_{1,\beta},
	\end{equation}
	for $\varepsilon_0$ sufficiently small, $|w_1|_{1,\beta},|w_2|_{1,\beta}\leq \varepsilon_0$ and a possibly different $\kappa_0$. Thus, Assumption A1 is satisfied with $k=1$. According to the definition of $\g'$, Assumption A2 is a dirrect application of Assumption H2. Assumption A3 is the same Assumption H4 for $k=1$. Now we have already checked all the Assumptions A1-A3.
\end{proof}

\subsection{Proof of propositions}\label{sec:proof-prop}
\begin{proof}[Proof of Proposition~\ref{prop:contract-t} ]
	Note that as $\sigma$ tends to zero, we have $\sigma \C w=o_{k,\beta}(1)$, $\g'[\sigma\C w]=\g'[\mathbf{0}]+o_{k,\beta}(1)$ and $\sigma\C\lambda\g'[\sigma\C w] =o_{k,\beta}(1)$ for all $|\lambda|\leq \sigma^{\alpha-1-\varepsilon}$ and $w\in\mathcal{B}$. This allow us to expand $\g(\sigma\C\lambda\g'[\sigma\C w])$ near the origin. We elaborate this expansion as follows.
	First, according to Assumption A1, we have that there exists a constant $\varepsilon_0$ such that for all $w\in\mathcal{B}$ and $\sigma\leq\varepsilon_0$,
	\begin{equation}\label{eq:gp-ex}
	\g'[\sigma\C w]=\g'[\mathbf{0}]
	+O_{k,\beta}(\sigma \C w).
	\end{equation}
	Second, with the aid of \eqref{eq:gp-ex} we have that for all $|\lambda_1|,|\lambda_2|\leq \sigma^{\alpha-1-\varepsilon}$ and $w\in\mathcal{B}$,
	\begin{equation}\label{eq:ex2}
	\sigma\C\lambda_1\g'[\sigma\C w]-\sigma\C\lambda_2\g'[\sigma\C w]=\sigma(\lambda_1-\lambda_2)\C \{\g'[\mathbf{0}]+O_{k,\beta}(\sigma \C w) \}.
	\end{equation}
	Thanks to Lemma~\ref{lemma:second-expansion} on page \pageref{lemma:second-expansion} and \eqref{eq:ex2}, we have that for all $|\lambda_1|,|\lambda_2|\leq \sigma^{\alpha-1-\varepsilon}$ and $w\in\mathcal{B}$,
	\begin{equation}\label{eq:ex4}
	\g(\sigma\C\lambda_1\g'[\sigma\C w])-\g(\sigma\C\lambda_2\g'[\sigma\C w])=\int_{\dom} \g'[\sigma\C\lambda_2\g'[\sigma\C w]](x)v(x)dx+O(|v|_{k,\beta}^2),
	\end{equation}
	where we define
	$$
	v(x)=\sigma\C\lambda_1\g'[\sigma\C w](x)-\sigma\C\lambda_2\g'[\sigma\C w](x).
	$$
	Setting $w$ as $\lambda_2\g'[\sigma\C w]$ in \eqref{eq:gp-ex}, we have
	\begin{equation}
	\label{eq:lam2}
	\g'[\sigma\C\lambda_2\g'[\sigma\C w]]= \g'[\mathbf{0}] + O_{k,\beta}(\sigma\C \lambda_2\g'[\sigma\C w] )=\g'[\mathbf{0}]+O_{k,\beta}(\sigma\lambda_2\g'[\mathbf{0}] ).
	\end{equation}
	The last equality in the above display is due to \eqref{eq:gp-ex} and the fact $O_{k,\beta}(\sigma\C w)=o_{k,\beta}(1)$.
	According to  \eqref{eq:ex2} and \eqref{eq:lam2}, we have
	\begin{eqnarray}
	&&\int_{\dom} \g'[\sigma\C\lambda_2\g'[\sigma\C w]](x)v(x)dx\notag\\
	& =& \sigma(\lambda_1-\lambda_2)\Big\{\int_\dom \C\g'[\mathbf{0}](x)\g'[\mathbf{0}](x)dx + O(\int_\dom\sigma^2 \lambda_2\g'[\mathbf{0}](x) \g'(x)dx)+O(\int_\dom \sigma \g'[\mathbf{0}](x)\C w (x)dx)\notag\\
	&&+ O(\sigma\lambda_2 \sigma \int_\dom\g'[\mathbf{0}](x)\C w(x)dx)\Big\}.\notag
	\end{eqnarray}
	Note that for $\lambda_2\in[-\sigma^{\alpha-1-\varepsilon},\sigma^{\alpha-1-\varepsilon}]$ the above expression is simplified as
	\begin{eqnarray}\label{eq:simp-ex3}
	\int_\dom\g'[\sigma\C\lambda_2\g'[\sigma\C w]](x)v(x)dx &=& \sigma(\lambda_1-\lambda_2)\Big\{ \int_\dom \C \g'[\mathbf{0}](x)\g'[\mathbf{0}](x)dx+O(\sigma^{\alpha-\varepsilon})\Big\}\notag\\
	&=&\sigma(\lambda_1-\lambda_2)\{ \K(\g'[\mathbf{0}])+O(\sigma^{\alpha-\varepsilon})\}.
	\end{eqnarray}
	Combining the above expression with \eqref{eq:ex4}, we have that for  $|\lambda_1|,|\lambda_2|\leq \sigma^{\alpha-1-\varepsilon}$ and $w\in\mathcal{B}$.
	\begin{equation*}
	\g(\sigma\C\lambda_1\g'[\sigma\C w])-\g(\sigma\C\lambda_2\g'[\sigma\C w])=\sigma(\lambda_1-\lambda_2)\{\K(\g'[\mathbf{0}])+O(\sigma^{\alpha-\varepsilon})\}+O(\sigma^2(\lambda_1-\lambda_2)^2),
	\end{equation*}
	which can be simplified as
	\begin{equation}\label{eq:g-exp}
	\g(\sigma\C\lambda_1\g'[\sigma\C w])-\g(\sigma\C\lambda_2\g'[\sigma\C w])=\sigma(\lambda_1-\lambda_2)\{
	\K(\g'[\mathbf{0}])+O(\sigma^{\alpha-\varepsilon})
	\}.
	\end{equation}
	Recall the definition of $T_{w}(\lambda)$, we plug the above expression into the difference $T_{w}(\lambda_1)-T_{w}(\lambda_2)$, and arrive at
	\begin{equation*}
	T_{w}(\lambda_1)-T_{w}(\lambda_2)= \lambda_1-\lambda_2 - \K(\g'[\mathbf{0}])^{-1}\sigma^{-1}\times\sigma(\lambda_1-\lambda_2)\{
	\K(\g'[\mathbf{0}])+O(\sigma^{\alpha-\varepsilon}) \},
	\end{equation*}
	which is simplified as
	\begin{equation*}
	T_{w}(\lambda_1)-T_{w}(\lambda_2)=-\K(\g'[\mathbf{0}])^{-1}(\lambda_1-\lambda_2)\times O(\sigma^{\alpha-\varepsilon}).
	\end{equation*}
	The above expression implies that for $|\lambda_1|,|\lambda_2|\leq \sigma^{\alpha-1-\varepsilon}$,
	\begin{equation}\label{eq:t-contract}
	T_{w}(\lambda_1)-T_{w}(\lambda_2)=(\lambda_1-\lambda_2)\times O(\sigma^{\alpha-\varepsilon}).
	\end{equation}
	This shows that $T_{w}(\lambda)$ is a contraction mapping for $\lambda\in[-\sigma^{\alpha-1-\varepsilon},\sigma^{\alpha-1-\varepsilon}]$.
	To see $T_{w}(\lambda)\in [-\sigma^{\alpha-1-\varepsilon},\sigma^{\alpha-1-\varepsilon}]$ for $\lambda\in[-\sigma^{\alpha-1-\varepsilon},\sigma^{\alpha-1-\varepsilon}]$ and $w\in\mathcal{B}$, we let $\lambda_2=0$ and $\lambda_1=\lambda$ in \eqref{eq:t-contract} and obtain that
	$$
	T_{w}(\lambda)-T_{w}(0)=\lambda O(\sigma^{\alpha-\varepsilon})=O(\sigma^{2\alpha-1-2\varepsilon}).
	$$
	Recall that $b=\kappa\sigma^{\alpha}$, and $T_{w}(0)=-\K(\g'[\mathbf{0}])^{-1}\sigma^{-1}b=-\kappa \K(\g'[\mathbf{0}])^{-1}\sigma^{\alpha-1}$.
	This implies
	\begin{equation}\label{eq:approx-tl}
	T_{w}(\lambda)=\kappa \K(\g'[\mathbf{0}])^{-1}\sigma^{\alpha-1}(1+o(1))\in [-\sigma^{\alpha-1-\varepsilon},\sigma^{\alpha-1-\varepsilon}]
	\end{equation} and concludes our proof.
\end{proof}

\begin{proof}[Proof of Proposition~\ref{prop:contract-w}]
	According to the definition of $\Xi$,
	\begin{equation*}
	\Xi[w_1]-\Xi[w_2]= \Lambda[w_1](\g'[\sigma\C w_1]-\g'[\sigma\C w_2])+(\Lambda[w_1]-\Lambda[w_2])\g'[\sigma\C w_2].
	\end{equation*}
	Therefore, we have
	\begin{equation}\label{eq:split-xi}
	|\Xi[w_1]-\Xi[w_2]|_{k,\beta}\leq |\Lambda[w_1]| \times|(\g'[\sigma\C w_1]-\g'[\sigma\C w_2])|_{k,\beta}+|\Lambda[w_1]-\Lambda[w_2]|\times|\g'[\sigma\C w_2]|_{k,\beta}.
	\end{equation}
	We establish upper bound for the first and second terms on the right-hand-side of the above inequality separately.
	To start with, according to Assumptions A1 and A3 that $\sup_{y\in\bar \dom}|C(\cdot,y)|_{k,2\beta}<\infty$, for $w_1,w_2\in\mathcal{B}$, we have
	\begin{equation}\label{eq:split-xi-first}
	|\Lambda[w_1]| \times|(\g'[\sigma\C w_1]-\g'[\sigma\C w_2])|_{k,\beta} = O(\sigma|\Lambda[w_1]||w_1-w_2|_{k,\beta})=O(\sigma^{\alpha})|w_1-w_2|_{k,\beta}.
	\end{equation}
	The second equality in the above expression is due to Lemma~\ref{lemma:approx-lam} on page~\pageref{lemma:approx-lam}.
	We proceed to the second term on the right-hand-side of \eqref{eq:split-xi}. 	
	Because $\Lambda[w]$ is the fixed point of $T_{w}(\cdot)$, we have
	$$
	T_{w_1}(\Lambda[w_1])=\Lambda[w_1]\mbox{ and }T_{w_2}(\Lambda[w_2])=\Lambda[w_2].
	$$
	Taking differencing between the above two equalities, we have
	$$
	T_{w_1}(\Lambda[w_1])-T_{w_2}(\Lambda[w_2])=\Lambda[w_1]-\Lambda[w_2].
	$$
	Adding and subtracting the term $T_{w_1}(\Lambda[w_2])$ in the above equality, we have
	$$
	\Lambda[w_1]-\Lambda[w_2]= T_{w_1}(\Lambda[w_1])-T_{w_1}(\Lambda[w_2])+T_{w_1}(\Lambda[w_2])-T_{w_2}(\Lambda[w_2]).
	$$
	Consequently,
	\begin{equation}\label{eq:bound-dif-lam}
	|\Lambda[w_1]-\Lambda[w_2]|\leq |T_{w_1}(\Lambda[w_1])-T_{w_1}(\Lambda[w_2])|+|T_{w_1}(\Lambda[w_2])-T_{w_2}(\Lambda[w_2])|.
	\end{equation}
	According to Proposition~\ref{prop:contract-t}, the first term on the right-hand-side of the above expression is bounded above by $O(\sigma^{\alpha-\varepsilon})|\Lambda[w_1]-\Lambda[w_2]|$.
	\begin{lemma}\label{lemma:t-contract-w}
		For all $|\lambda|= O(\sigma^{\alpha-1})$ and $w_1,w_2\in\mathcal{B}$, we have
		$$
		|T_{w_1}(\lambda)-T_{w_2}(\lambda)|= O(\sigma^{\alpha}) |w_1-w_2|_{k,\beta}.
		$$
	\end{lemma}
	According to Lemma~\ref{lemma:t-contract-w}, the second term on the right-hand-side of \eqref{eq:bound-dif-lam} is bounded above by $O(\sigma^{\alpha})|w_1-w_2|_{k,\beta}.$ Therefore, we have
	$$
	|\Lambda[w_1]-\Lambda[w_2]|\leq O(\sigma^{\alpha-\varepsilon})|\Lambda[w_1]-\Lambda[w_2]|+O(\sigma^{\alpha})|w_1-w_2|_{k,\beta}.
	$$
	Consequently, we have that for $w_1,w_2\in\mathcal{B}$,
	\begin{equation}\label{eq:lam-lip}
	|\Lambda[w_1]-\Lambda[w_2]|= O(\sigma^{\alpha})|w_1-w_2|_{k,\beta}.
	\end{equation}
	According to \eqref{eq:gp-ex},
	$$|\g'[\sigma\C w_2]|_{k,\beta}=O(1).$$
	The above approximation and \eqref{eq:lam-lip} give
	\begin{equation*}
	|\Lambda[w_1]-\Lambda[w_2]|\times |\g'[\sigma\C w_2]|_{k,\beta} = O(\sigma^{\alpha})|w_1-w_2|_{k,\beta}.
	\end{equation*}
	Combining the above display with \eqref{eq:split-xi} and \eqref{eq:split-xi-first}, we complete our proof.
\end{proof}

\begin{proof}[Proof of Proposition~\ref{prop:opt}]

	(i) is a direct application of Proposition~\ref{prop:contract-w}, contraction mapping theorem and the KKT condition~\eqref{KKT:w}. We proceed to the proof of (ii).
	Because $\xi^*$ is the fixed point of $\Xi$ in $\mathcal{B}$, we have
	$$
	\Xi[\xi^*]=\Lambda(\Xi^*)\g'[\sigma\C \xi^*]= \kappa \K(\g'[\mathbf{0}])^{-1}\sigma^{\alpha-1}(1+o(1))(\g'[\mathbf{0}]+O_{k,\beta}(\sigma \xi^*))=(1+o_{k,\beta}(1))\frac{\kappa \g'[\mathbf{0}]}{\K(\g'[\mathbf{0}])}.
	$$
	To obtain the second equality in the above display, we use approximation in Lemma~\ref{lemma:approx-lam} on page~\pageref{lemma:approx-lam} and \eqref{eq:gp-ex}.
\end{proof}
\section*{Acknowledgement}
Jingchen Liu is partially  supported by the National Science Foundation (SES-1323977,  IIS-1633360) and  Army Grant (W911NF-15-1-0159). Jianfeng Lu is partially supported by National Science Foundation (DMS-1454939). Xiang Zhou acknowledges the support from Hong Kong General Research Fund (109113, 11304314, 11304715).

%

\bibliographystyle{apalike}
\bibliography{PDE,MaterialsFailure,Asym1D18}

\begin{thebibliography}{}

\bibitem[Adler, 1981]{Adl81}
Adler, R. (1981).
\newblock {\em The Geometry of Random Fields}.
\newblock Wiley, Chichester, U.K.; New York, U.S.A.

\bibitem[Adler et~al., 2008]{ABL08}
Adler, R., Blanchet, J., and Liu, J. (2008).
\newblock Efficient simulation for tail probabilities of {G}aussian random
  fields.
\newblock In {\em Proceeding of Winter Simulation Conference}.

\bibitem[Adler et~al., 2012]{ABL09}
Adler, R., Blanchet, J., and Liu, J. (2012).
\newblock Efficient {M}onte {C}arlo for large excursions of {G}aussian random
  fields.
\newblock {\em Ann. Appl. Probab.}, 22(3):1167--1214.

\bibitem[Adler and Taylor, 2007]{AdlTay07}
Adler, R. and Taylor, J. (2007).
\newblock {\em Random fields and geometry}.
\newblock Springer.

\bibitem[Azais and Wschebor, 2008]{AW08}
Azais, J.~M. and Wschebor, M. (2008).
\newblock A general expression for the distribution of the maximum of a
  {G}aussian field and the approximation of the tail.
\newblock {\em Stochastic Processes and their Applications}, 118(7):1190--1218.

\bibitem[Azais and Wschebor, 2009]{AW09}
Azais, J.~M. and Wschebor, M. (2009).
\newblock {\em Level sets and extrema of random processes and fields}.
\newblock Wiley, Hoboken, N.J.

\bibitem[Bear and Verruijt, 1987]{Bear1987Modeling-Ground}
Bear, J. and Verruijt, A. (1987).
\newblock {\em Modeling Groundwater Flow and Pollution}.
\newblock D. Reidel Publishing Company, Holland.

\bibitem[Berman, 1985]{Berman85}
Berman, S.~M. (1985).
\newblock An asymptotic formula for the distribution of the maximum of a
  {G}aussian process with stationary increments.
\newblock {\em Journal of Applied Probability}, 22(2):454--460.

\bibitem[Borell, 1975]{Borell1975}
Borell, C. (1975).
\newblock The brunn-minkowski inequality in gauss space.
\newblock {\em Inventiones mathematicae}, 30:207--216.

\bibitem[Borell, 2003]{Bor03}
Borell, C. (2003).
\newblock The {Ehrhard} inequality.
\newblock {\em Comptes Rendus Mathematique}, 337(10):663--666.

\bibitem[Charbeneau, 2000]{Charbeneau2000}
Charbeneau, R.~J. (2000).
\newblock {\em Groundwater Hydraulics and Pollutant Transport}.
\newblock Prentice Hall.

\bibitem[Cirel'son et~al., 1976]{Cirel'son1976}
Cirel'son, B.~S., Ibragimov, I.~A., and Sudakov, V.~N. (1976).
\newblock Norms of gaussian sample functions.
\newblock In Maruyama, G. and Prokhorov, J.~V., editors, {\em Proceedings of
  the Third Japan --- USSR Symposium on Probability Theory}, pages 20--41.
  Springer Berlin Heidelberg, Berlin, Heidelberg.

\bibitem[Freeze, 1975]{Freeze-lognormal}
Freeze, R. (1975).
\newblock A stochastic-conceptual analysis of one-dimensional groundwater flow
  in nonuniform homogeneous media.
\newblock {\em Water Resour. Res.}, 11.

\bibitem[Gilbarg and Trudinger, 2015]{gilbarg2015elliptic}
Gilbarg, D. and Trudinger, N.~S. (2015).
\newblock {\em Elliptic partial differential equations of second order}.
\newblock springer.

\bibitem[Landau and Shepp, 1970]{LS70}
Landau, H.~J. and Shepp, L.~A. (1970).
\newblock Supremum of a {G}aussian process.
\newblock {\em Sankhya-the Indian Journal of Statistics Series A},
  32(Dec):369--378.

\bibitem[Ledoux and Talagrand, 1991]{LT91}
Ledoux, M. and Talagrand, M. (1991).
\newblock {\em Probability in {B}anach spaces: isoperimetry and processes}.
\newblock Ergebnisse der Mathematik und ihrer Grenzgebiete 3. Folge, Bd. 23.
  Springer-Verlag, Berlin ; New York.

\bibitem[Li et~al., 2015]{li2015tail}
Li, X., Liu, J., and Xu, G. (2015).
\newblock On the tail probabilities of aggregated lognormal random fields with
  small noise.
\newblock {\em Mathematics of Operations Research}.

\bibitem[Li et~al., 2016]{LiLiuXu2016}
Li, X., Liu, J., and Xu, G. (2016).
\newblock On the tail probabilities of aggregated lognormal random fields with
  small noise.
\newblock {\em Mathematics of Operations Research}, 41(1):236--246.

\bibitem[Liu, 2012]{Liu10}
Liu, J. (2012).
\newblock Tail approximations of integrals of {G}aussian random fields.
\newblock {\em Ann. Probab.}, 40:1069--1104.

\bibitem[Liu et~al., 2015]{SISC2015}
Liu, J., Lu, J., and Zhou, X. (2015).
\newblock Efficient rare event simulation for failure problems in random media.
\newblock {\em SIAM Journal on Scientific Computing}, 37(2):A609--A624.

\bibitem[Liu and Xu, 2012]{LiuXu11}
Liu, J. and Xu, G. (2012).
\newblock Some asymptotic results of {G}aussian random fields with varying mean
  functions and the associated processes.
\newblock {\em Ann. Stat.}, 40:262--293.

\bibitem[Liu and Zhou, 2013]{LiuZhou13}
Liu, J. and Zhou, X. (2013).
\newblock On the failure probability for one dimensional random material under
  delta external force.
\newblock {\em Commun. Math. Sci.}, 11(2):499 -- 521.

\bibitem[Liu and Zhou, 2014]{LiuZhou14}
Liu, J. and Zhou, X. (2014).
\newblock Extreme analysis of a random ordinary differential equation.
\newblock {\em J. Appl. Probab.}

\bibitem[Marcus and Shepp, 1970]{MS70}
Marcus, M.~B. and Shepp, L.~A. (1970).
\newblock Continuity of {G}aussian processes.
\newblock {\em Transactions of the American Mathematical Society},
  151(2):377--391.

\bibitem[Piterbarg, 1996]{Pit95}
Piterbarg, V.~I. (1996).
\newblock {\em Asymptotic methods in the theory of {G}aussian processes and
  fields}.
\newblock American Mathematical Society, Providence, R.I.

\bibitem[Sudakov and Tsirelson, 1974]{ST74}
Sudakov, V. and Tsirelson, B. (1974).
\newblock Extremal properties of half spaces for spherically invariant
  measures.
\newblock {\em Zap. Nauchn. Sem. LOMI}, 45:75--82.

\bibitem[Sun, 1993]{Sun93}
Sun, J.~Y. (1993).
\newblock Tail probabilities of the maxima of {G}aussian random-fields.
\newblock {\em Annals of Probability}, 21(1):34--71.

\bibitem[Talagrand, 1996]{TA96}
Talagrand, M. (1996).
\newblock Majorizing measures: The generic chaining.
\newblock {\em Annals of Probability}, 24(3):1049--1103.

\bibitem[Taylor and Adler, 2003]{TayAdl03}
Taylor, J. and Adler, R. (2003).
\newblock Euler characteristics for {G}aussian fields on manifolds.
\newblock {\em Annals of Probability}, 31(2):533--563.

\bibitem[Taylor et~al., 2005]{TTA05}
Taylor, J., Takemura, A., and Adler, R.~J. (2005).
\newblock Validity of the expected {E}uler characteristic heuristic.
\newblock {\em Annals of Probability}, 33(4):1362--1396.

\bibitem[Tsirelson et~al., 1976]{CIS}
Tsirelson, B., Ibragimov, I., and Sudakov, V. (1976).
\newblock Norms of {G}aussian sample functions.
\newblock {\em Proceedings of the Third Japan-USSR Symposium on Probability
  Theory (Tashkent, 1975)}, 550:20--41.

\bibitem[Xu et~al., 2014]{KdVUQ2014}
Xu, G., Lin, G., and Liu, J. (2014).
\newblock Rare-event simulation for the stochastic korteweg--de vries equation.
\newblock {\em SIAM/ASA Journal on Uncertainty Quantification}, 2(1):698--716.

\end{thebibliography}

\newpage
\appendix
\section{Proof of supporting lemmas}\label{sec:proof-lemmas}
\begin{proof}[Proof of Lemma~\ref{lemma:localization}]

Note that the event $\{\xi-\CC \xi^*\notin\mathcal{B} \}=\{
|\xi-\CC \xi^*|_{k,\beta}>\sigma^{\alpha-1-\varepsilon}
\}$ implies the event $\{|\xi|>\sigma^{\alpha-1-\varepsilon}-|\xi^*|_{k,\beta} \}$. According to Proposition~\ref{prop:opt}, $|\xi^*|_{k,\beta}=O(\sigma^{\alpha-1})$. Thus,
\begin{equation}\label{eq:event-bc}
\{\xi-\CC \xi^*\notin\mathcal{B} \}\subset\{
|\xi|_{k,\beta} >\varepsilon_0 \sigma^{\alpha-1-\varepsilon}
\},
\end{equation}
for a positive constant $\varepsilon_0$ and $\sigma$ sufficiently small.
Recall the definition
\begin{equation*}
|\xi|_{k,\beta}=\sum_{l=1}^k \sup_{|\gamma|=l}\sup_{x\in \bar \dom} |D^{\gamma}\xi(x)|+\sup_{|\gamma|=k}[D^{\gamma}\xi]_{\beta}.
\end{equation*}	
Consequently,
\begin{eqnarray*}
\{\xi-\CC\xi^*\notin \mathcal{B}\}&\subset&\bigcup_{l=1}^k\{\sup_{|\gamma|=l}
\sup_{x\in\bar \dom}|D^{\gamma}\xi(x)|>\frac{\sigma^{\alpha-1-\varepsilon}}{k+1}
\}\bigcup \{
\sup_{|\gamma|=l}[D^{\gamma}\xi]_{\beta}>\frac{\sigma^{\alpha-1-\varepsilon}}{k+1}
\}\\
&=&\bigcup_{l=1}^k\bigcup_{|\gamma|=l}\{\sup_{x\in\bar \dom}|D^{\gamma}\xi(x)|>\frac{\sigma^{\alpha-1-\varepsilon}}{k+1}
\}\bigcup_{|\gamma|=l} \{
[D^{\gamma}\xi]_{\beta}>\frac{\sigma^{\alpha-1-\varepsilon}}{k+1}
\}.
\end{eqnarray*}
The equality in the above display is due to the fact that $\{\sup_{l=1}^m X_l\geq \eta \}= \cup_{l=1}^m\{
X_l\geq \eta
\}$ for any random variable $X_l$, $l=1,...,m$ and constant $\eta$. According to the above display, we arrive at a upper bound of probability.
\begin{equation}\label{eq:split-prob}
\PP(\xi-\CC \xi^*\notin\mathcal{B})\leq
\sum_{l=1}^k\sum_{|\gamma|=l}\PP(\sup_{x\in\bar \dom}|D^{\gamma}\xi(x)|>\frac{\sigma^{\alpha-1-\varepsilon}}{k+1}) +\sum_{|
\gamma|=l}\PP([D^{\gamma}\xi]_{\beta}>\frac{\sigma^{\alpha-1-\varepsilon}}{k+1}).
\end{equation}
We establish upper bounds for $\PP(\sup_{x\in\bar
  \dom}|D^{\gamma}\xi(x)|>\frac{\sigma^{\alpha-1-\varepsilon}}{k+1})$ and
$\PP([D^{\gamma}\xi]_{\beta}>\frac{\sigma^{\alpha-1-\varepsilon}}{k+1})$
separately. We first analyze the term $\PP(\sup_{x\in\bar
  \dom}|D^{\gamma}\xi(x)|>\frac{\sigma^{\alpha-1-\varepsilon}}{k+1})$. We
will need the following lemma, known as the Borell-TIS inequality,
which was proved independently by \cite{Borell1975} and
\cite{Cirel'son1976}.
	\begin{lemma}[Borell-TIS inequality]\label{lemma:borel}
		Let $g(x)$ be a centered and almost surely bounded Gaussian
		random field. Then, $\EE\sup_{x\in \dom} |g(x)|<\infty$. Furthermore, for any $t>\EE\sup_{x\in \dom} |g(x)|$, we have
		$$
		\PP\Big(\sup_{x\in \dom}|g(x)|-\EE\sup_{x\in \dom} |g(x)|>t\Big)\leq 2 \exp\left\{-\frac{t^2}{2\sup_{x\in \dom}\Var(g(x))}\right\}.
		$$
	\end{lemma}
According to Lemma~\ref{lemma:borel}, we have that for all $|\gamma|\leq k$, $\EE\sup_{x\in\bar \dom}|D^{\gamma}\xi(x)|<\infty$ and
\begin{equation}\label{eq:tail-bound-dgamma}
\PP(\sup_{x\in\bar \dom}|D^{\gamma}\xi(x)|>\frac{\sigma^{\alpha-1-\varepsilon}}{k+1})\leq 2\exp\left\{-\frac{\sigma^{2\alpha-2-2\varepsilon}}{8(k+1)^2\sup_{x\in\bar \dom}C_{D^{\gamma}\xi}(x,x)}\right\},
\end{equation}
for $\sigma$ sufficiently small such that $\sigma^{2\alpha-2-2\varepsilon}> 2\EE\sup_{x\in\bar \dom}|D^{\gamma}\xi(x)|$, and $C_{D^{\gamma}\xi}$ is defined \eqref{eq:def-cd}.
According to Assumption A3,  there exists a constant $\kappa_0$ such that for all $|\gamma|\leq k$,
$$
\sup_{x\in\bar \dom}C_{D^{\gamma}\xi}(x,x)\leq \sup_{y\in\bar \dom}|C_{D^{\gamma}\xi}(\cdot,y)|_{\beta}<\kappa_0.
$$
The above display together with \eqref{eq:tail-bound-dgamma} give
\begin{equation*}
\PP(\sup_{x\in\bar \dom}|D^{\gamma}\xi(x)|>\frac{\sigma^{\alpha-1-\varepsilon}}{k+1}) \leq 2\exp\{-\frac{\sigma^{2\alpha-2-2\varepsilon}}{8(k+1)^2\kappa_0}\}
\end{equation*}
Combine this with \eqref{eq:split-prob}, we have
\begin{equation}\label{eq:split-prob2}
\PP(\xi-\CC \xi^*\notin\mathcal{B})\leq \kappa_0\exp\{-\frac{\sigma^{2\alpha-2-2\varepsilon}}{8(k+1)^2\kappa_0}\}+\sum_{|
	\gamma|=l}\PP([D^{\gamma}\xi]_{\beta}>\frac{\sigma^{\alpha-1-\varepsilon}}{k+1}),
\end{equation}
for a possibly different $\kappa_0$ such that $\kappa_0\geq 2\mathrm{Card}\{\gamma: |\gamma|\leq k \}$.
We proceed to establishing upper bounds for $ \PP([D^{\gamma}\xi]_{\beta}>\frac{\sigma^{\alpha-1-\varepsilon}}{k+1})$, $|\gamma|=k$.
Recall that
$$
[D^{\gamma}\xi]_{\beta}=\sup_{x,y\in \bar \dom, x\neq y}\frac{|D^{\gamma}\xi(x)-D^{\gamma}\xi(y)|}{|x-y|^{\beta}}.
$$
Motivated by this definition, we define another centered Gaussian random field double indexed by $x,y\in\bar \dom$
	\begin{equation}
	g(x,y)=\left\{
	\begin{array}{lcl}\frac{D^{\gamma}\xi(x)-D^{\gamma}\xi(y)}{|x-y|^{\beta}}&\mbox{ for }& x\neq y\\
	0 &\mbox{ for }& x=y
	\end{array}.
	\right.
	\end{equation}
According to Assumption A3 $\xi\in C^{k,\beta}(\bar \dom)$ almost surely. Thus, $g(\cdot,\cdot)$ is bounded almost surely. According to Lemma~\ref{lemma:borel}, we have that $\EE{\sup_{x,y\in\bar \dom,x\neq y}|g(x,y)|}<\infty$, and
\begin{equation*}
\PP(\sup_{x,y\in\bar \dom}|g(x,y)|>\frac{\sigma^{\alpha-1-\varepsilon}}{k+1})\leq 2 \exp\left\{
-\frac{\sigma^{2\alpha-2-2\varepsilon}}{8(k+1)^2\sup_{x,y\in\bar \dom}\Var g(x,y)}
\right\},
\end{equation*}
for $\sigma$ sufficiently small such that $\sigma^{2\alpha-2-2\varepsilon}> 2\EE\sup_{x,y\in\bar \dom}|g(x,y)|$.
The variance of $g(x,y)$ in the above expression is bounded above as follows.
\begin{equation*}
  \begin{aligned}
\Var g(x,y) &= |x-y|^{-2\beta}\{
C_{D^{\gamma}\xi}(x,x)-C_{D^{\gamma}\xi}(x,y)+C_{D^{\gamma}\xi}(y,y)-C_{D^{\gamma}\xi}(x,y)\} \\
& \leq [C_{D^{\gamma}\xi}(x,\cdot)]_{2\beta	} +[C_{D^{\gamma}\xi}(y,\cdot)]_{2\beta	},
\end{aligned}
\end{equation*}
which is bounded above by a constant $\kappa_0$ according to Assumption A3.
Thus, we have
\begin{equation*}
\PP(\sup_{x,y\in\bar \dom}|g(x,y)|>\frac{\sigma^{\alpha-1-\varepsilon}}{k+1})\leq 2 \exp\left\{
-\frac{\sigma^{2\alpha-2-2\varepsilon}}{8(k+1)^2\kappa_0}
\right\}.
\end{equation*}
Note that $[D^{\gamma}\xi]_{\beta}=\sup_{x,y\in\bar \dom}|g(x,y)|$. Therefore, the above display is equivalent to
\begin{equation}\label{eq:bound-dg}
\PP([D^{\gamma}\xi]_{\beta}>\frac{\sigma^{\alpha-1-\varepsilon}}{k+1})\leq 2 \exp\left\{
-\frac{\sigma^{2\alpha-2-2\varepsilon}}{8(k+1)^2\kappa_0}
\right\}.
\end{equation}
We conclude our proof by combining the above inequality with \eqref{eq:split-prob2}.
\end{proof}

\begin{proof}[Proof of Lemma~\ref{lemma:approx-lam}]
	Because $\Lambda[w]$ is a fixed point of $T_{w}(\cdot)$, this lemma is a direct application of \eqref{eq:approx-tl}.
\end{proof}

\begin{proof}[Proof of Lemma~\ref{lemma:second-expansion}]
	We define a function $h:[0,1]\to\mathbb{R}$,
	$$
	h(s)=\g(w_2+s(w_1-w_2))-\g(w_2)-s\int_\dom \g'[w_2](x)\{w_1(x)-w_2(x)\}dx.
	$$
	Notice that $h(0)=0$ and $h(1)=\g(w_1)-\g(w_2)-\int_\dom \g'[w_2](x)(w_1(x)-w_2(x))dx$. Apply mean value theorem to $h$, we have
	\begin{equation}\label{eq:mean-value}
	\g(w_1)-\g(w_2)-\int_\dom \g'[w_2](x)(w_1(x)-w_2(x))dx=h(1)-h(0)=h'(\tilde{s}),
	\end{equation}
	for some $\tilde{s}\in[0,1]$.
	According to the definition of Fr\'echet derivative, it is easy to check that
	$$
	h'(s)=s\int_\dom\{\g'[w_1+s(w_1-w_2)](x)-\g'[w_2](x)\}(w_1(x)-w_2(x))dx.
	$$
	Furthermore, we have
	\begin{eqnarray*}
		&&\Big|s\int_\dom\{\g'[w_1+s(w_1-w_2)](x)-\g'[w_2](x)\}(w_1(x)-w_2(x))dx\Big|\\
		&\leq& meas(\dom) |w_1-w_2|_{0}\times|\g'[w_1+s(w_1-w_2)](x)-\g'[w_2](x)|_0\\
		&\leq& meas(\dom) |w_1-w_2|_{0}\times \kappa_G |w_1-w_2|_{k,\beta}\\
		&\leq& meas(\dom) |w_1-w_2|_{k,\beta}^2.
	\end{eqnarray*}
	Here, $meas(\dom)$ is the Lebesgue measure of the set $\dom$, the second inequality is due to Assumption A1, and the third inequality is due to the fact that $w$, $|w|_0\leq |w|_{k,\beta}$. Combine the above inequality and \eqref{eq:mean-value} we obtain the desired result.
\end{proof}

\begin{proof}[Proof of Lemma~\ref{lemtail}]
	We prove the lemma by induction.
	We first prove this lemma for the case where $k=0$ and $\beta>0$.
	We consider the conditional random field $\{\xi(x),x\in \bar \dom \mid Z_1= z \}$. 
	It can be shown that there exists a  continuous Gaussian random field,  
	denoted by $\{\chi(x),x\in \bar \dom\}$, who has the same distribution as $\{\xi(x),x\in \bar \dom \mid Z_1= z \}$ and belongs to $C^{\beta}(\bar \dom)$ almost surely.
	The mean and covariance function of $\chi(x)$ satisfy
	\begin{eqnarray*}
		\mu_{\chi}(x)&=& \Var(Z_1)^{-1}\Cov(Z_1,\xi(x)) z = \Var(Z_1)^{-1} \int_\dom \xi^*(y)C(x,y)dy,\\
		C_{\chi}(x,y)&=& C(x,y)-\Var(Z_1)^{-1}\Cov(\xi(x),Z_1)\Cov(\xi(y),Z_1)\\&=& C(x,y)-\Var(Z_1)^{-1} \int_\dom \xi^*(z)C(x,z)dz \int_\dom \xi^*(z)C(y,z)dz.
	\end{eqnarray*}
	According to the expression \eqref{eq:var-z1} and  $\sup_{y\in\bar \dom}|C(\cdot,y)|_{2\beta}\in <\infty$, we have that,
	\begin{equation}\label{eq:mean-order}
	|\mu_{\chi}|_{\beta}=O(\sigma^{1-\alpha}z) \mbox{ and } 	\sup_{y\in \bar \dom}|C_{\chi}(,y)|_{2\beta}<\infty.
	\end{equation}
	Let $\zeta(x)=\chi(x)-\mu_{\chi}(x)$ be a centered Gaussian random field. Then event $\{
	|\chi|_{\beta}^2>\frac{\Lambda(\xi^*)z}{\sigma}
	\}$ implies that $\{|\zeta|_{\beta}>(\frac{\Lambda(\xi^*)z}{\sigma})^{\frac{1}{2}}-|\mu_{\chi}|_{\beta}\}$. Furthermore, according to \eqref{eq:mean-order} and Lemma~\ref{lemma:approx-lam} on page~\pageref{lemma:approx-lam}, we have
	$$
	\{
	|\chi|_{\beta}^2>\frac{\Lambda(\xi^*)z}{\sigma}
	\}\subset\{
	|\zeta|_{\beta}>\varepsilon_0 \sigma^{\frac{\alpha}{2}-1}\sqrt{z}-O(\sigma^{1-\alpha}z)
	\}.
	$$
	Because $z\leq \kappa_0 \sigma^{2\alpha-2-\varepsilon}$, we have $\sigma^{\frac{\alpha}{2}-1}\sqrt{z}-O(\sigma^{1-\alpha}z) \geq \varepsilon_0 \sigma^{\frac{\alpha}{2}-1}\sqrt{z}$ for a possibly different $\varepsilon_0$. Therefore,
	$$
	\{
	|\chi|_{\beta}^2>\frac{\Lambda(\xi^*)z}{\sigma}
	\}\subset\{
	|\zeta|_{\beta}>\varepsilon_0 \sigma^{\frac{\alpha}{2}-1}\sqrt{z}
	\}.
	$$
	Consequently, we have
	\begin{equation}\label{eq:bound-chi}
	\PP(|\chi|_{\beta}^2>\frac{\Lambda(\xi^*)z}{\sigma})\leq \PP(|\zeta|_{\beta}>\varepsilon_0 \sigma^{\frac{\alpha}{2}-1}\sqrt{z}).
	\end{equation}
	According to the definition of the norm
	$
	|\zeta|_{\beta}=\sup_{x\in\bar \dom}|\zeta(x)|+[\zeta]_{\beta}.
	$
	Therefore, an upper bound for  \eqref{eq:bound-chi} is
	\begin{equation}\label{eq:split}
	\PP(|\chi|_{\beta}^2>\frac{\Lambda(\xi^*)z}{\sigma})\leq \PP(\sup_{x\in \bar \dom} |\zeta(x)|\geq\frac{\varepsilon_0}{2}\sigma^{\frac{\alpha}{2}-1}\sqrt{z}) + \PP([\zeta]_{\beta}\geq\frac{\varepsilon_2}{2}\sigma^{\frac{\alpha}{2}-1}\sqrt{z}).
	\end{equation}
	We will present upper bounds for the first and second terms in the above display separately. We start with the first term. Because $\zeta$ is a centered and continuous Gaussian random field, with the aid of Lemma~\ref{lemma:borel}, we have that $\EE\sup_{x\in \bar \dom}|\zeta(x)|<\infty$ and
	\begin{equation*}
	\PP(\sup_{x\in\bar \dom}|\zeta(x)|>\frac{\varepsilon_0}{2}\sigma^{\frac{\alpha}{2}-1}\sqrt{z})\leq 2\exp\left\{-\frac{\varepsilon_0^2 \sigma^{\alpha-2}z}{32\sup_{x\in \bar \dom}\Cov_{\chi}(x,x)}\right\},
	\end{equation*}
	for $\sigma$ and $z$ such that $\varepsilon_0\sigma^{\frac{\alpha}{2}-1}\sqrt{z}>2\EE \sup_{x\in \bar \dom} |\zeta(x)|$. Because $z\geq \sigma^{\varepsilon}$, $\varepsilon_0\sigma^{\frac{\alpha}{2}-1}\sqrt{z}>2\EE \sup_{x\in \bar \dom} |\zeta(x)|$ is satisfied for $\sigma$ sufficiently small. Consequently, for $\sigma$ sufficiently small, we have
	\begin{equation}\label{eq:bound-sup}
	\PP(\sup_{x\in\bar \dom}|\zeta(x)|>\frac{\varepsilon_0}{2}\sigma^{\frac{\alpha}{2}-1}\sqrt{z})< e^{-\varepsilon_0 \sigma^{\alpha-2}z}
	\end{equation}
	for a sufficiently small and possibly different $\varepsilon_0$.
	We proceed to the second term on the right-hand-side of \eqref{eq:split}.
	Because $\zeta\in C^{\beta}(\bar \dom)$ almost surely, we obtain an upper bound for $\PP([\zeta]_{\beta}>\frac{\varepsilon_0}{2}\sigma^{\frac{\alpha}{2}-1}\sqrt{z})$ using similar arguments as those for \eqref{eq:bound-dg} on page \pageref{eq:bound-dg}
	\begin{equation}\label{eq:bound-zeta}
	\PP([\zeta]_{\beta}>\frac{\varepsilon_0}{2}\sigma^{\frac{\alpha}{2}-1}\sqrt{z})<2e^{-\varepsilon_0 \sigma^{\alpha-2}z},
	\end{equation}
	for $\sigma$ sufficiently small and a positive constant $\varepsilon_0$.
	Combine \eqref{eq:split}, \eqref{eq:bound-sup} and \eqref{eq:bound-zeta}, we have
	\begin{equation}\label{eq:bound-chi-beta}
		\PP(|\chi|_{\beta}^2>\frac{\Lambda(\xi^*)z}{\sigma})<2e^{-\varepsilon_0 \sigma^{\alpha-2}z}.
	\end{equation}
Recall that $\chi$ has the same distribution as $\{
\zeta(x):x\in\bar \dom|Z_1=z
\}$, thus \eqref{eq:bound-chi-beta} implies
	\begin{equation}\label{eq:z1-z}
	\PP(|\xi|_{\beta}^2>\frac{\Lambda(\xi^*)z}{\sigma}|Z_1=z)< 2e^{-\varepsilon_0 \sigma^{\alpha-2}z}.
	\end{equation}
	Using similar arguments, we have that for $\sigma$ sufficiently small
	\begin{equation}\label{eq:bond-betaplus}
	\PP(|\xi|_{\beta}^2>\frac{\Lambda(\xi^*)z}{\sigma}|Z_1=-z)< 2e^{-\varepsilon_0 \sigma^{\alpha-2}z}.
	\end{equation}
	Combing the above inequality with \eqref{eq:z1-z}, we have
	$$
	\PP(|\xi|_{\beta}^2>\frac{\Lambda(\xi^*)z}{\sigma}|Z_1=z)+\PP(|\xi|_{\beta}^2>\frac{\Lambda(\xi^*)z}{\sigma}|Z_1=z)<4e^{-\varepsilon_0 \sigma^{\alpha-2}z}< e^{-\varepsilon_0'\sigma^{\alpha-2}z}
	$$
	for $\varepsilon_0'<\varepsilon_0$ and $\sigma$ sufficiently small.
	This completes our proof for the case where $k=0$ and $\beta>0$.
	For the case $k=0$ and $\beta=0$, $|\xi|_{\beta}=|\xi|_0$. With similar proof as those for \eqref{eq:bound-sup}, we have
	\begin{equation}\label{eq:bond-beta-0}
	\PP(|\xi|_0^2>\frac{\Lambda(\xi^*)z}{\sigma}|Z_1=z)\leq \PP(\sup_{x\in \bar \dom} |\zeta(x)|\geq\frac{\varepsilon_0}{2}\sigma^{\frac{\alpha}{2}-1}\sqrt{z}) <2e^{-\varepsilon_0\sigma^{\alpha-2}z}.
	\end{equation}
	We also have similar results conditional on $Z_1=-z$. Therefore, for $\beta=0$ we also have
	$$
	\PP(|\xi|_{\beta}^2>\frac{\Lambda(\xi^*)z}{\sigma}|Z_1=z)+\PP(|\xi|_{\beta}^2>\frac{\Lambda(\xi^*)z}{\sigma}|Z_1=z)<4e^{-\varepsilon_0 \sigma^{\alpha-2}z}.
	$$
	This completes our proof for the case that $k=0$.
	We now proceed to prove the lemma for $k\geq 1$.
	Assuming that for $k=m$,
	\begin{equation}
	\label{eq:induction-assumption}
	\PP(|\xi|_{k,\beta}^2>\frac{\Lambda(\xi^*)z}{\sigma}|Z_1=z)+\PP(|\xi|_{k,\beta}^2>\frac{\Lambda(\xi^*)z}{\sigma}|Z_1=z)<e^{-\varepsilon_0 \sigma^{\alpha-2}z}
	\end{equation}
	for some positive constant $\varepsilon_0$ that is independent with $\sigma$ and $z$ but possibly depend on $k$.
	We will prove that the following inequality holds for $\sigma$ sufficiently small and a positive constant $\varepsilon_0$,
	\begin{equation}\label{eq:induction-goal}
	\PP(|\xi|_{m+1,\beta}^2>\frac{\Lambda(\xi^*)z}{\sigma}|Z_1=z)+\PP(|\xi|_{m+1,\beta}^2>\frac{\Lambda(\xi^*)z}{\sigma}|Z_1=z)<e^{-\varepsilon_0' \sigma^{\alpha-2}z}.
	\end{equation}
	According to the definition of the norm $|\cdot|_{m+1,\beta}$, we know that for $\beta>0$ $$|\xi|_{m+1,\beta}=|\xi|_m+\sup_{|\gamma|=m+1}\sup_{x\in\bar \dom}|D^{\gamma}\xi(x)|+\sup_{|\gamma|=m+1}[D^{\gamma}\xi]_{\beta}.$$
	Therefore,
	\begin{eqnarray*}
		&&\Big\{
		|\xi|_{m+1,\beta}^2\geq \frac{\Lambda(\xi^*)z}{\sigma}
		\Big\}\\
		&\subset& \{
		|\xi|_{m}^{2}\geq\frac{\Lambda(\xi^*)z}{2\sigma}
		\}\bigcup\Big(\bigcup_{|\gamma|=m+1}\bigcup_{|\gamma'|=m+1}\{(\sup_{x\in\bar \dom}|D^{\gamma}\xi(x)|+[D^{\gamma'}\xi]_{\beta})^2\geq\frac{\Lambda(\xi^*)z}{2\sigma} \}\Big).
	\end{eqnarray*}
	Consequently, we arrive at an upper bound
	\begin{eqnarray}
	&&\PP\Big(
	|\xi|_{m+1,\beta}^2\geq \frac{\Lambda(\xi^*)z}{\sigma}
	|Z_1=z\Big)\notag\\
	&\leq&\PP(|\xi|_{m}^{2}\geq\frac{\Lambda(\xi^*)z}{2\sigma}|Z_1=z)\notag\\
	&&+\sum_{|\gamma|=m+1}\sum_{|\gamma'|=m+1}\PP\Big((\sup_{x\in\bar \dom}|D^{\gamma}\xi(x)|+[D^{\gamma'}\xi]_{\beta})^2\geq\frac{\Lambda(\xi^*)z}{2\sigma} |Z_1=z\Big),\label{eq:split-induction}
	\end{eqnarray}
	We present upper bounds for the first and second terms on the right-hand-side of the above display separately.
	For the first term, according to \eqref{eq:induction-assumption}, we have
	\begin{equation}\label{eq:split-induction-term1}
	\PP(|\xi|_{m}^{2}\geq\frac{\Lambda(\xi^*)z}{2\sigma}|Z_1=z)\leq e^{-\varepsilon_0 \sigma^{\alpha-2}z}.
	\end{equation}
	For the second term, notice that
	\begin{eqnarray}
	&&\{(\sup_{x\in\bar \dom}|D^{\gamma}\xi(x)|+[D^{\gamma'}\xi]_{\beta})^2\geq\frac{\Lambda(\xi^*)z}{2\sigma}\}\notag\\
	&=&\{\sup_{x\in\bar \dom}|D^{\gamma}\xi(x)|+[D^{\gamma'}\xi]_{\beta}\geq\sqrt{\frac{\Lambda(\xi^*)z}{2\sigma}}\}\notag\\
	&\subset&\{
	\sup_{x\in\bar \dom}|D^{\gamma}\xi(x)|\geq \frac{1}{2}\sqrt{\frac{\Lambda(\xi^*)z}{2\sigma}}
	\}\cup\{[D^{\gamma'}\xi]_{\beta}\geq\frac{1}{2}\sqrt{\frac{\Lambda(\xi^*)z}{2\sigma}}  \}.\notag
	\end{eqnarray}
	Therefore,
	\begin{eqnarray}
	&&\PP\Big((\sup_{x\in\bar \dom}|D^{\gamma}\xi(x)|+[D^{\gamma'}\xi]_{\beta})^2\geq\frac{\Lambda(\xi^*)z}{2\sigma}|Z_1=z\Big)\notag\\
	&\leq& \PP\Big(\sup_{x\in\bar \dom}|D^{\gamma}\xi(x)|\geq \frac{1}{2}\sqrt{\frac{\Lambda(\xi^*)z}{2\sigma}}|Z_1=z\Big)+\PP\Big([D^{\gamma'}\xi]_{\beta}\geq\frac{1}{2}\sqrt{\frac{\Lambda(\xi^*)z}{2\sigma}}|Z_1=z\Big).\label{eq:bound-mplus}
	\end{eqnarray}
	Now we present upper bounds for the two terms on the right-hand-side of the above inequality for $\gamma$ and $\gamma'$ such that $|\gamma|=m+1$ and $|\gamma'|=m+1$. To do so, we consider a continuous Gaussian random field
	$\chi_{1}$ that belongs to $C^{\beta}(\bar \dom)$ almost surely, and it has the same distribution as $\{
	D^{\gamma}\xi(x),x\in\bar \dom|Z_1=z
	\}$.
	\begin{lemma}\label{lemma:higher-order-covariance}
		Let $C_{\chi_1}(s,t)=\EE\chi_1(s)\chi_1(t)$ and $\mu_{\chi_1}(t)= \EE \chi_1(t) $, then we have
		$$
		|\mu_{\chi_1}|_{\beta}=O(\sigma^{1-\alpha}z)\mbox{ and } \sup_{y \in \bar \dom}|C_{\chi_1}(\cdot,y)|<\infty.
		$$
		The above expressions are uniform in $\gamma$ for $|\gamma|=m+1$.
	\end{lemma}
	Notice that the above lemma has the same form as \eqref{eq:mean-order}, so with similar arguments as those for  \eqref{eq:bound-sup}, we have
	\begin{equation}\label{eq:bound-chi1-sup}
	\PP\Big(\sup_{x\in\bar \dom}|D^{\gamma}\chi_1(x)|\geq \frac{1}{2}\sqrt{\frac{\Lambda(\xi^*)z}{2\sigma}}\Big)\leq e^{-\varepsilon_0 \sigma^{\alpha-2}z}.
	\end{equation}
	Also, similar as arguments before \eqref{eq:bound-zeta}, we have
	\begin{equation}\label{eq:bound-chi1-beta}
	\PP\Big([D^{\gamma'}\chi_1]_{\beta}\geq \frac{1}{2}\sqrt{\frac{\Lambda(\xi^*)z}{2\sigma}}\Big)\leq e^{-\varepsilon_0 \sigma^{\alpha-2}z}.
	\end{equation}
	Combining \eqref{eq:bound-chi1-sup} and \eqref{eq:bound-chi1-beta} and \eqref{eq:bound-mplus}, we have
	$$
	\PP\Big((\sup_{x\in\bar \dom}|D^{\gamma}\xi(x)|+[D^{\gamma'}\xi])^2\geq\frac{\Lambda(\xi^*)z}{2\sigma}|Z_1=z\Big)\leq 2e^{-\varepsilon_0 \sigma^{\alpha-2}z}.
	$$
	Combining the above display with \eqref{eq:split-induction} and \eqref{eq:split-induction-term1}, we have
	\begin{equation*}
	\PP\Big(
	|\xi|_{m+1,\beta}^2\geq \frac{\Lambda(\xi^*)z}{\sigma}
	|Z_1=z\Big)\leq e^{-\varepsilon_0\sigma^{\alpha-2}z},
	\end{equation*}
	for $\sigma$ sufficiently small and a possibly different constant $\varepsilon_0$.
	Similarly, conditional on $Z_1=-z$, we have
	\begin{equation*}
		\PP\Big(
		|\xi|_{m+1,\beta}^2\geq \frac{\Lambda(\xi^*)z}{\sigma}
		|Z_1=-z\Big)\leq e^{-\varepsilon_0\sigma^{\alpha-2}z}.
	\end{equation*}
	Thus,
	$$
		\PP\Big(
		|\xi|_{m+1,\beta}^2\geq \frac{\Lambda(\xi^*)z}{\sigma}
		|Z_1=z\Big)+	\PP\Big(
		|\xi|_{m+1,\beta}^2\geq \frac{\Lambda(\xi^*)z}{\sigma}
		|Z_1=-z\Big)\leq2 e^{-\varepsilon_0\sigma^{\alpha-2}z},
	$$
	and
	we complete the proof for \eqref{eq:induction-goal} for the case where $\beta>0$. For $\beta=0$, $|\xi|_{m+1}=|\xi|_m+\sup_{|\gamma|=m+1}\sup_{x\in\bar \dom}|D^{\gamma}\xi(x)|$.
	We obtain the proof for the case where $\beta=0$ by ignoring all the $[D^{\gamma'}\xi]_{\beta}$ terms in the proof for the case where $\beta>0$. This completes the induction.
\end{proof}

\begin{proof}[Proof of Lemma~\ref{lemma:regularity}]
	According to Theorem 6.14 in \cite{gilbarg2015elliptic}, we have that the PDE \eqref{eq:gPDE} has a unique solution in $C^{2,\beta}(\bar \dom)$. Denote this solution by $u_{a,f}$, then according to Theorem 6.6 in \cite{gilbarg2015elliptic}, we have the upper bound
	\begin{equation*}
	|u_{a,f}|_{2,\beta}\leq \kappa(\delta,M,d,\dom)(|u_{a,f}|_0+|f|_0).
	\end{equation*}
	We conclude the proof with the following  upper bound provided by Theorem 3.7 in\\ \cite{gilbarg2015elliptic},
	$$
	|u_{a,f}|_0\leq \kappa_0 |f|_0
	$$
	for a constant $\kappa_0$ depending only on the domain $\dom$ and $|a|_{1}$.
\end{proof}

\begin{proof}[Proof of Lemma~\ref{lemma:stability}]
	According to the definition of $u_{a_1,f_1}$ and $u_{a_2,f_2}$, we have that
	$$
	-\nabla\cdot(a_1(x)\nabla u_{a_1,f_1}(x))=f_1\mbox{ and }	-\nabla\cdot(a_2(x)\nabla u_{a_2,f_2}(x))=f_2.
	$$
	Taking difference between the above two equalities, we have
	\begin{equation*}
	-\nabla\cdot(a_1\nabla u_{a_1,f_1})+\nabla\cdot(a_2(x)\nabla u_{a_2,f_2})=f_1(x)-f_2(x) \mbox{ for } x\in \dom.
	\end{equation*}
	Rearranging terms in the above expression, we have
	\begin{equation*}
	-\nabla\cdot\Big(a_2(x)\nabla(u_{a_2,f_2}(x)-u_{a_1,f_1}(x))\Big)=f_2(x)-f_1(x)-\nabla\cdot\{(a_1(x)-a_2(x))\nabla u_{a_1,f_1}(x) \}.
	\end{equation*}
	Therefore, $\bar{u}=u_{a_2,f_2}-u_{a_1,f_1}\in C^{2,\beta}(\bar \dom)$ is a solution to the elliptic PDE
	\begin{equation*}
	\begin{cases}
	- \nabla \cdot (a_2(x) \nabla \bar u (x)) = \bar{f}(x) & \text{for  } x\in\dom;  \\
	\bar u (x)= 0 & \text{for } x\in\partial \dom,
	\end{cases}
	\end{equation*}
	where $\bar f(x)=f_2(x)-f_1(x)-\nabla\cdot\{(a_1(x)-a_2(x))\nabla u_{a_1,f_1}(x) \}$.
	According to Lemma~\ref{lemma:regularity}, we have
	\begin{equation}\label{eq:bound3}
	|u_{a_2,f_2}-u_{a_1,f_1}|_{2,\beta}\leq \kappa(\delta,M,d,\dom) |\bar f|_{\beta}.
	\end{equation}	
	We further establish an upper bound for $|\bar f|_{\beta}$,
	\begin{eqnarray}\label{eq:bound4}
	|\bar f|_{\beta}\leq |f_2-f_1|_{\beta}+ |a_2-a_1|_{1,\beta}|u_{a_1,f_1}|_{2,\beta}.
	\end{eqnarray}
	According to Lemma~\ref{lemma:regularity},
	$$
	|u_{a_1,f_1}|_{2,\beta}\leq \kappa(\delta,M,d,\dom) |f_1|_{\beta}.
	$$
	Combining this with \eqref{eq:bound3} and \eqref{eq:bound4}, we have
	\begin{equation*}
	|u_{a_2,f_2}-u_{a_1,f_1}|_{2,\beta}\leq \kappa(\delta,M,d,\dom) \{|f_2-f_1|_{\beta}+\kappa(\delta,M,d,\dom)|a_2-a_1|_{1,\beta}|f_1|_{\beta}\}.
	\end{equation*}
	We complete the proof by setting $\tilde{\kappa}(\delta,M,d,\dom)=\max(\kappa(\delta,M,d,\dom),\kappa(\delta,M,d,\dom)^2)$.
\end{proof}

		\begin{proof}[Proof of Lemma~\ref{lemma:t-contract-w}]
		We take difference between $T_{w_1}(\lambda)$ and $T_{w_2}(\lambda)$,
		\begin{equation*}
		T_{w_1}(\lambda)-T_{w_2}(\lambda) = -\K(\g'[\mathbf{0}])^{-1}\sigma^{-1}\{\g(\sigma\C\lambda\g'[\sigma \C w_1])-\g(\sigma\C \lambda\g'[\sigma\C w_2])\}
		\end{equation*}
		Therefore,
		\begin{equation}\label{eq:bound5}
		|T_{w_1}(\lambda)-T_{w_2}(\lambda)|=O(\sigma^{-1}\{|\g(\sigma\C \lambda\g'[\sigma \C w_1])-\g(\sigma\C \lambda\g'[\sigma\C w_2])|\}).
		\end{equation}
		According to Lemma~\ref{lemma:second-expansion}, we have
		\begin{eqnarray*}
			&&\g(\sigma\lambda\C\g'[\sigma \C w_1])-\g(\sigma\lambda\C\g'[\sigma\C w_2])\\
			&=&
			\sigma\lambda\int_{\dom}\g'[\sigma\lambda\C\g'[\sigma\C  w_2]](x)\C\{\g'[\sigma \C w_1](x)-\g'[\sigma \C w_2](x)\}dx\\
			&& + O(\sigma^2\lambda^2|\g'[\sigma \C w_1]-\g'[\sigma \C w_2]|_{k,\beta}^2).
		\end{eqnarray*}
		According to \eqref{eq:gp-ex} and Assumption A1, the above display can be further simplified as
		$$
		\g(\sigma\lambda\C\g'[\sigma w_1])-\g(\sigma\lambda\C\g'[\sigma w_2])= O(\sigma\lambda |\g'[\sigma w_1]-\g'[\sigma w_2]|_{k,\beta}),
		$$
		which is further simplified as
		$$
		\g(\sigma\lambda\C\g'[\sigma w_1])-\g(\sigma\lambda\C\g'[\sigma w_2])=O(\sigma\lambda \sigma|w_1-w_2|_{k,\beta}).
		$$
		The above expression and \eqref{eq:bound5} give
		\begin{equation*}
		|T_{w_1}(\lambda)-T_{w_2}(\lambda)|=O(\sigma\lambda|w_1-w_2|_{k,\beta}) = O(\sigma^{\alpha})|w_1-w_2|_{k,\beta}.
		\end{equation*}
		The last inequality in the above expression is due to $\lambda=O(\sigma^{\alpha-1})$.
\end{proof}

\begin{proof}[Proof of Lemma~\ref{lemma:higher-order-covariance}]
We need the next lemma for the current proof.
\begin{lemma}\label{lemma:cov}
We define the covariance  function $$ C_{D^{\gamma}\xi,\xi}(x,y) = \Cov(D^{\gamma}\xi(x),\xi(y)).
	$$
	Then  $\sup_{y\in\bar \dom}|C_{D^{\gamma}\xi,\xi}(\cdot,y)|_{2\beta}<\infty$ for all $|\gamma|\leq k$ under Assumption A3.
\end{lemma}
Now we compute the mean and covariance of $\chi_1$.
$$
\mu_{\chi_1}(x)=\E[D^{\gamma}\xi(x)|Z_1=z] = \Var(Z_1)^{-1}\Cov(D^{\gamma}\xi(x), Z_1)z = \Var(Z_1)^{-1}\int_{\dom}  C_{D^{\gamma}\xi,\xi}(x,y)\xi^*(y)dtz,
$$
and
\begin{eqnarray*}
C_{\chi_1}(x,y)
 &=& C_{D^{\gamma}\xi}(x,y)-\Var(Z_1)^{-1}\Cov(D^{\gamma}\xi(x),Z_1)\Cov(D^{\gamma}\xi(y),Z_1)\\ &=&  C_{D^{\gamma}\xi}(x,y)- \frac{\int_{\dom}  C_{D^{\gamma}\xi,\xi}(x,r)\xi^*(y)dr \int_{\dom}  C_{D^{\gamma}\xi,\xi}(y,r)\xi^*(y)dr}{\Var(Z_1)}.
\end{eqnarray*}
Recall that $\Var(Z_1)\geq \varepsilon_0\sigma^{2\alpha-2}$ for some positive constant $\varepsilon$, and $|\xi^*|_{k,\beta}=O(\sigma^{\alpha-1})$. With the aid of Lemma~\ref{lemma:cov}, we simplify the mean and covariance of $\chi_1$.
\begin{equation*}
|\mu_{\chi_1}|_{\beta}= O(\sigma^{2-2\alpha}|\xi^*|_{0} \sup_{y}|C_{D^{\gamma}\xi,\xi}(\cdot,y)|_{2\beta} z)=O(\sigma^{1-\alpha}z),
\end{equation*}
and
\begin{equation*}
\sup_{y\in \bar \dom}|C_{\chi_1}(\cdot,y)|_{2\beta}=O(\sup_{y\in \bar \dom}|C_{D^{\gamma}\xi}(\cdot,y)|_{2\beta}+\sigma^{2-2\alpha} |\xi^*|_{0}^2 \sup_{y\in\bar \dom}|C_{D^{\gamma}\xi,\xi}(\cdot,y)|_{2\beta}^2)= O(1).
\end{equation*}
\end{proof}

\begin{proof}[Proof of Lemma~\ref{lemma:cov}]
	We will use induction to prove that for all $l=0,1,...,k$, $|\gamma|=l$,
	\begin{equation}\label{eq:induction-want}
 \sup_{y\in\bar \dom}|C_{D^{\gamma}\xi,\xi}(\cdot,y)|_{2\beta}<\infty.
	\end{equation}
	To start with, for $l=0$ and $|\gamma|=l$, \eqref{eq:induction-want} holds because of Assumption A3 and
	$$
 C_{D^{\gamma}\xi,\xi}(s,t) = C(s,t).
	$$
	Suppose that for all $|\gamma'|=l$,
	\begin{equation}\label{eq:induction-suppose}
 \sup_{y\in\bar \dom}|C_{D^{\gamma'}\xi,\xi}(\cdot,y)|_{k-l,2\beta}<\infty.
	\end{equation}
	For $|\gamma|=l+1$, we want to show that
	\begin{equation}\label{eq:induction-goal-Dg}
 \sup_{y\in\bar \dom}|C_{D^{\gamma}\xi,\xi}(\cdot,y)|_{k-l-1,2\beta}<\infty.
	\end{equation}
	Without loss of generality, we assume that $\gamma=(\gamma_1,...,\gamma_d)$ and $\gamma_1\geq 1$. Let $e_1=(1,...,0)$ be a $d$-dimensional basis vector, and $\gamma'=\gamma-e_1$, then $|\gamma'|=l$.
	We compute $C_{D^{\gamma}\xi,\xi}$.
	\begin{eqnarray*}
		C_{D^{\gamma}\xi,\xi}(x,y) &=&  \lim_{\varepsilon_1\to 0} \Cov(\frac{D^{\gamma'}\xi(x+\varepsilon_1 e_1)-D^{\gamma'}\xi(x)}{\varepsilon_1}, \xi(y))\\
		&=&\lim_{\varepsilon_1\to0}\varepsilon_1^{-1}\{C_{D^{\gamma'}\xi,\xi}(x+\varepsilon_1 e_1,y)-C_{D^{\gamma'}\xi,\xi}(x,y)\}\\
		&=& \frac{\partial}{\partial x_1} C_{D^{\gamma'}\xi} (x,y).
	\end{eqnarray*}
	Consequently,
	\begin{equation*}
	|C_{D^{\gamma}\xi,\xi}(\cdot,y)|_{k-l-1,2\beta}=|\frac{\partial}{\partial x_1} C_{D^{\gamma'}\xi} (\cdot,y)|_{k-l-1,2\beta}\leq |C_{D^{\gamma'}\xi} (\cdot,y)|_{k-l,2\beta}.
	\end{equation*}
	Thus,
	$$
	\sup_{y\in\bar \dom}	|C_{D^{\gamma}\xi,\xi}(\cdot,y)|_{k-l-1,2\beta}\leq \sup_{y\in\bar \dom}|C_{D^{\gamma'}\xi} (\cdot,y)|_{k-l,2\beta}<\infty.
	$$
	The second inequality of the above display is due to \eqref{eq:induction-suppose}.
The lemma is proved by induction.
\end{proof}

\end{document}